\documentclass[a4paper,11pt,twoside]{article}
\usepackage[a4paper,top=5cm,bottom=5cm,left=4cm,right=4cm]{geometry}

\usepackage[latin1]{inputenc}
\usepackage[T1]{fontenc}
\usepackage[english]{babel}

\usepackage{sansmath}
\usepackage{libertine}
\usepackage{libertinust1math}
\usepackage[scaled]{beramono}

\usepackage{calrsfs}
\usepackage{enumitem}

\setlist{nolistsep}
\usepackage{mathtools}
\usepackage{mathrsfs}
\usepackage{mfirstuc}
\usepackage{multicol}

\usepackage[overload]{textcase}
\usepackage{tasks}
\settasks{counter-format=tsk[1].}
\usepackage{bbm}
\usepackage{amsmath}
\usepackage{amssymb}
\usepackage{amsfonts}
\usepackage{amsthm}
\usepackage{wasysym}
\usepackage{dsfont}
\usepackage{graphicx}
\usepackage{wrapfig}
\usepackage[outdir=./]{epstopdf}
\usepackage[labelfont={bf,it},textfont=it]{caption}
\usepackage[hidelinks]{hyperref}
\usepackage{widetable}
\usepackage[scr=pxtx, scrscaled=1]{mathalfa}
\usepackage{xfrac}
\usepackage{soul}
\usepackage{tikz}
\usepackage{setspace}
\usepackage{yfonts}
\usepackage{url}

\newcommand{\grad}{\nabla}

\newcommand*{\IE}{\mathbb{E}}
\newcommand*{\IR}{\mathbb{R}}
\newcommand*{\IRi}{\IR\cup\{+\infty\}}

\newcommand*{\IN}{\mathbb{N}}
\newcommand*{\states}{\mathcal{E}}
\renewcommand*{\marks}{\mathcal{S}}
\newcommand*{\confs}{\mathcal{M}}
\newcommand*{\x}{\mathsfbf{x}}
\newcommand*{\y}{\mathsfbf{y}}

\newcommand{\abs}[1]{\lvert#1\rvert}
\newcommand{\norm}[1]{\lVert #1\rVert}
\newcommand{\refmark}{\mathbf{R}}

\newcommand{\tg}{\textswab}

\newcommand{\msb}{\mathsfbf}
\newcommand*{\Energy}{H}
\newcommand*{\PairEnergy}{E_\Phi}

\newcommand{\tempered}{\text{\normalfont temp}}
\newcommand*{\Gibbs}{\mathcal{G}_{z,\beta}(\Energy)}
\newcommand*{\GibbsT}{\mathcal{G}^{\tempered}_{z,\beta}(\Energy)}
\newcommand{\CondEn}[2]{E_\Phi\left(#1\,\vert\, #2\right)}
\newcommand*{\Banachz}{\mathbb{X}_{\constanz}}
\newcommand*{\Banach}{\mathbb{X}_{\constan}}
\newcommand*{\K}{\msb{K}}
\DeclareSymbolFont{CMletters}{OML}{cmm}{m}{it}
\DeclareMathSymbol{\xi}{\mathord}{CMletters}{"18}
\renewcommand*{\Xi}{\varXi}
\renewcommand*{\epsilon}{\varepsilon}
\renewcommand*{\theta}{\vartheta}
\renewcommand*{\Theta}{\varTheta}
\renewcommand*{\Lambda}{\varLambda}
\renewcommand*{\Delta}{\varDelta}

\newcommand*{\1}{\mathds{1}}
\newcommand{\hyp}[1]{$(\mathcal{H}_{\text{#1}})$}
\newcommand{\corr}[2]{\rho_{#1}^{(#2)}}
\newcommand*{\RuelleC}{\mathscr{C}(\beta)}
\newcommand*{\constanz}{\msb{c}_z}
\newcommand*{\constan}{\msb{c}}
\newcommand*{\zRuelle}{\tg{z}_{Ru}(\beta)}
\newcommand*{\zCrit}{\tg{z}_{\text{crit}}(\beta)}

\newcommand*{\interval}{1}
\newcommand*{\stabconst}{B}
\newcommand{\constPsi}{A_\Psi}
\newcommand*{\sigmab}{\sigma}
\newcommand{\defeq}{%
	\mathrel{ \vcenter{\baselineskip0.5ex \lineskiplimit0pt \hbox{\scriptsize.} \hbox{\scriptsize.}} } =}
\renewcommand{\eqdef}{%
	\mathrel{= \vcenter{\baselineskip0.5ex \lineskiplimit0pt \hbox{\scriptsize.} \hbox{\scriptsize.}} } }
\theoremstyle{plain}
\theoremstyle{definition}
\newtheorem{Definition}{\sffamily Definition}[section]
\newtheorem{Example}{\sffamily Example}
\newtheorem{Excont}{\sffamily Example}

\newtheorem{Lemma}[Definition]{\sffamily Lemma}
\newtheorem{Corollary}[Definition]{\sffamily Corollary}
\newtheorem{Proposition}[Definition]{\sffamily Proposition}
\newtheorem{Assumption}{\sffamily Assumption}

\newtheorem{Theorem}{\sffamily Theorem}
\theoremstyle{remark}
\newtheorem{Remark}[Definition]{\normalfont\sffamily Remark}
\newtheorem{Remarks}[Definition]{\normalfont\sffamily Remarks}
\newtheorem*{Example*}{\normalfont\sffamily Example}

\newtheorem{innercustomgeneric}{\customgenericname}
\providecommand{\customgenericname}{}
\newcommand{\newcustomtheorem}[2]{%
  \newenvironment{#1}[1]
  {%
   \renewcommand\customgenericname{\normalfont\bfseries\sffamily #2}%
   \ifx&#1&%
   \renewcommand\theinnercustomgeneric{}
   \else \renewcommand\theinnercustomgeneric{\theAssumption ##1}
   \fi
   \innercustomgeneric
  }
  {\endinnercustomgeneric}
}

\newcustomtheorem{customthm}{Theorem}
\newcustomtheorem{AssumptionOpt}{Assumption}

\usepackage{sectsty}
\allsectionsfont{\sffamily\large}
\subsectionfont{\sffamily}
\newcommand*\myrule[1][.25\textwidth]{%
	\tikz {\path [fill, draw] (0,0) [out=0, in=180] to +(.5*#1,1pt) [out=0, in=180] to +(.5*#1,-1pt) [out=180, in=0] to +(-.5*#1,-1pt) [out=180, in=0] to cycle;}}

\makeatletter
\def\@maketitle{%
  \newpage
  \null
  \vskip 2em%
  \begin{center}%
  \let \footnote \thanks
    \begin{minipage}{.7\textwidth}
    	\begin{center}
    		\sffamily\Large\textbf{{ \@title }}\par
    	\end{center}
    \end{minipage}
    \vskip 1.5em%
    {\large
      \lineskip .5em%
      \vskip -.8em
      \vspace{-1em} 
      \begin{tabular}[t]{c}%
        {\@author}
      \end{tabular}\par}%
    \vskip .5em%
    \par
  \vskip .7em
  \vspace{-1.5em} 
  \myrule[.15\textwidth]
  \vskip 1.5em
  \end{center}%
  }
\makeatother

\renewenvironment{abstract}{%
\vskip -2em
\hfill\begin{center}\small
\begin{minipage}{0.75\textwidth}}
{\par\noindent
\end{minipage}
\end{center}
\vskip 2.5em}
\makeatletter

\usepackage{caption}
\newcommand{\comment}[1]{#1}

\usepackage{authblk}

\title{Gibbs point processes on path space: existence, cluster expansion and uniqueness}
\author{Alexander Zass}
\affil{Universit\"at Potsdam, WIAS Berlin
\\
{\small \texttt{zass@wias-berlin.de}}
}
\usepackage{fancyhdr}
\fancyheadoffset{0cm}

\usepackage{blindtext}
\usepackage{parskip}

\begin{document}

\maketitle

\begin{abstract}
	We present general existence and uniqueness results for marked models with pair interactions, exemplified through Gibbs point processes on path space. More precisely, we study a class of infinite-dimensional diffusions under Gibbsian interactions, in the context of marked point configurations: the starting points belong to $\IR^d$, and the marks are the paths of Langevin diffusions. We use the entropy method to prove existence of an infinite-volume Gibbs point process and use cluster expansion tools to provide an explicit activity domain in which uniqueness holds.
	\bigskip
	
	\smallskip
	
\end{abstract}


\section*{Introduction}

\comment{The motivating example for this work is a system of infinitely many Langevin diffusions in interaction: through the lens of Gibbs point process theory, we first see each diffusion -- starting in $x\in\IR^d$ and with displacement $\big(m(s),\ s\in[0,\interval]\big)$ -- as a marked point $\x = (x,m)\in\mathcal{E}\defeq\IR^d\times C_0$, where $C_0$ is the normed space of continuous paths starting at $0$. We then consider an infinite collection of them, in terms of a marked point process on $\IR^d\times C_0$. Finally,
we add, on the space of marked point configurations, a Gibbsian energy functional $H$, with (finite but) not uniformly bounded interaction range. In such a framework, the questions of existence and uniqueness of Gibbs point processes are far from trivial; in particular, we note how the norm of the random marks is a priori unbounded.}

\comment{The interest in point processes on path space can be traced back to the works of R. Feynman and M. Ka\v{c} for quantum particle systems (\cite{kac,feynman}). In this context, by applying a functional integration to the operator $e^{-\beta H}$, D. Ruelle obtained in \cite{ruelle_1969} the integration over a random number of Brownian bridges of length $\beta$. This representation was interpreted by J. Ginibre (\cite{ginibre}, and more recently, e.g. \cite{rafler_thesis,ueltschi_06}) as a point processes on the space of composite loops.
The first description of a quantum system in terms of a marked point processes is due to K.-H. Fichtner (\cite{fichtner_1991}), in the non-interacting case: the marks, elements of $\IR^d$, are given by the sequence of positions (at discrete time points) of the gas particles. In the interacting Gibbsian framework, in \cite{ACK11}, the authors use marked point processes to provide a variational formula for the free energy of a system of repulsive bosons.}
\comment{Using the Gibbsian framework to describe the interaction between open paths (and not only loops) is then highly sought-after, as Gibbs point processes could be seen to be the solutions of a variational problem. This is connected to the open question of Bose--Einstein condensation (see the discussion in \cite{ACK11}). A complete analysis of this phenomenon is missing at the moment, and this work hopes to contribute towards an understanding of the Gibbsian nature of processes on path space.}

\comment{We are interested here in open trajectories, i.e. when the final position is not a priori fixed. The best setting for this is that of marked points, where the starting location of the path is a point $x\in\IR^d$ and its mark is given by a (centred) path $m\in C_0$.}

In this work we present general existence and uniqueness results (Theorem \ref{thm:diff:existence} and Theorem \ref{thm:diff:uniqueness}) for marked models with pair interactions, which we exemplify through Gibbs point processes on path space (Section \ref{sec:diff:path}). 
In this setting, we consider interactions between two trajectories given as time integrals of (instantaneous) pair potentials which are usual for classical systems in $\IR^d$, like the Lennard--Jones pair potential. \comment{Note that we allow for potentials that are not (fully) repulsive, which induces technical complexities. Nevertheless, we require -- in order to satisfy the stability conditions that are needed in the method (Assumption \ref{hyp:diff:1}) -- a hard-core component near the origin (see Example \ref{ex:diff:0}).}

After presenting the general setting in Section \ref{sec:diff:setting}, we tackle the existence question in Section \ref{sec:diff:existence}, via the Dobrushin--Lanford--Ruelle description of Gibbs point processes. Under some stability assumptions for the energy functional $\Energy$, we are able to prove (in Theorem \ref{thm:diff:existence}) the existence, for any activity $z$ and inverse temperature $\beta$, of at least one infinite-volume marked Gibbs point process $P^{z,\beta}$ with energy functional $\Energy$, by applying the entropy method presented for the general marked setting in \cite{roelly_zass_2020}.

Moreover, we also show that, for any $N\geq 1$, the $N$-point correlation function $\rho_N$ of these Gibbs point processes satisfy a (point-dependent) \emph{Ruelle bound} of the following form: there exists a function $\constan\colon\mathcal{E}\to\IR_+$ such that, for almost any finite path configuration $(\x_1,\dots,\x_N)\in\mathcal{E}^N$,
\begin{equation*}
	\rho_N(\x_1,\dots,\x_N)\leq \prod_{i=1}^N \constan(\x_i).
\end{equation*}

In Sections \ref{sec:diff:RBcorrelation} and \ref{sec:diff:KS} we present, as a novel result, an explicit activity domain where uniqueness of the Gibbs point process holds. This is obtained with the approach of cluster expansion and the Kirkwood--Salsburg equations -- a method which was first developed for lattice systems in the 1980s (see e.g. \cite{malyshev_1980}) and then extended to the continuous case (see e.g. \cite{malyshev_minlos_1991,nehring_poghosyan_zessin_2012}). 

In the case of unmarked continuous point processes, the technique relies on considering a series expansion of the correlation functions. As presented by D. Ruelle in \cite{ruelle_1969}, one first shows that the correlation functions of a Gibbs point process can be expressed as an absolutely converging series of cluster terms, and then proves uniqueness by considering a system of integral equations -- the so-called Kirkwood--Salsburg equations -- that the correlation functions satisfy. In fact, these equations can be reformulated as a fixed-point problem for an operator $\K_z$ in an appropriately chosen Banach space, having therefore a unique solution.

The cluster expansion approach is actually well adapted to the marked setting. Indeed, S. Poghosyan and D. Ueltschi develop, in \cite{poghosyan_ueltschi_2009}, abstract \comment{finite-volume} techniques that can be used both in the classical and in the marked setting, under assumptions of so-called \emph{modified regularity} of the interaction. These assumptions and techniques are further developed in \cite{poghosyan_zessin_2020} by S. Poghosyan and H. Zessin, proving uniqueness of infinite-volume Gibbs point processes for potentials satisfying a certain stability condition (which they refer to as \emph{Penrose stability}). Some similar result is presented by S. Jansen in \cite{jansen_2019}, but making strong use of the repulsive nature of the interaction she considers.

A key tool in the Kirkwood--Salsburg proof of uniqueness is a Ruelle bound for the correlation functions of any Gibbs point process, as this shows that they belong to a certain Banach space $\Banach$ (see Subsection \ref{sec:diff:RBnonunif}).
The techniques used in \cite{poghosyan_zessin_2020} are too restrictive if applied in our setting of unbounded marks,
so we use here a different approach, inspired by the work \cite{kuna_1999} of T. Kuna: in Section \ref{sec:diff:RBcorrelation}, under a different set of assumptions than that of Section \ref{sec:diff:existence}, we rely on some tree-graph estimates, to prove a Ruelle bound for the correlation functions of infinite-volume Gibbs point processes.
In particular, under an additional regularity assumption for the interaction potential $\Phi$, we show that there exists an activity threshold $\zRuelle>0$ such that, for any $z\in(0,\zRuelle)$, the correlation functions $\corr{N}{P}$ of any Gibbs point process $P$ with activity $z$ and inverse temperature $\beta$ satisfy a Ruelle bound as above, but with a constant $\constanz$: for almost any $(\x_1,\dots,\x_N)\in\mathcal{E}^N$, $\corr{N}{P}(\x_1,\dots,\x_N)\leq \constanz^N$.

In Section \ref{sec:diff:KS}, after showing that there exists an activity threshold $\zCrit>0$ such that, for any $z\in(0,\zCrit)$, the norm of the Kirkwood--Salsburg operator $\K_z$ in $\Banach$ is strictly smaller than $1$, we show that the associated equations have a unique solution and obtain the following uniqueness domain (in Theorem \ref{thm:diff:uniqueness}): for any $\beta>0$ and $z\in(0,\zCrit)$, there exists a unique infinite-volume \comment{marked} Gibbs point process $P$ with activity $z$ and inverse temperature $\beta$, associated to the energy functional $\Energy$.

An important note is that, by using the entropy method, and not Kirkwood--Salsburg, as an existence method, we are able to obtain the existence for any activity $z>0$, and have therefore a restriction on the activity regime that is only due to the requirements for uniqueness.

\section{The Gibbsian setting}

Before presenting the general setting of this work, we present its main motivation, that is a marked point process representation for interacting diffusions.

\subsection{A point process on path space}\label{sec:diff:path}

We consider infinitely-many independent gradient diffusions and add a dependence between them by introducing an \emph{interaction energy} in the context of marked Gibbs point processes. 

The basic mathematical object of this work is the following Langevin dynamics on $\IR^d$:
\begin{equation}\label{eq:diff:langevin}
	dX(s) = dB(s) - \frac{1}{2}\nabla V\big(X(s)\big)ds,\quad s\in[0,\interval],
\end{equation}
where $B$ is a standard $\IR^d$-valued Brownian motion, and $V\colon\IR^d\rightarrow\IR$ is a smooth potential satisfying, outside of some compact subset of $\IR^d$,
\begin{equation}\label{eq:diff:VgradV}
  \exists \delta',b_1,b_2>0, \quad  V(x) \geq b_1 \lvert x\rvert^{d+\delta'} \text{ and } \Delta V(x) -\frac{1}{2}\lvert \grad V (x)\rvert^2 \leq -b_2\vert x\rvert^{2+2\delta'}.
\end{equation}
It is a known result (see e.g. \cite{royer_2007}) that, under these conditions, there exists a unique solution to the SDE \eqref{eq:diff:langevin}, which generates an \emph{ultracontractive} semigroup (see \cite{kavian_1993, davies_1989}). Moreover, for any $\delta<\delta'/2$,
\begin{equation}\label{eq:diff:deltaDiff}
	\IE\left[e^{\sup_{s\in[0,\interval]}\abs{X(s)-X(0)}^{d+2\delta}}\right]<+\infty.
\end{equation}

For the rest of this work, let $\delta>0$ be fixed.


Consider now that any (continuous) path $\x$ on $[0,1]$ can be decomposed into its initial location $x$ and a (shifted) path $m$ starting from $0$. In other words, we identify $\x$ with a marked point given by the pair $(x,m)\in \mathcal{E}\defeq \IR^d\times C_0$, where $\marks$ is the space of continuous paths on $[0,\interval]$ starting at $0$. The space $C_0$, endowed with the norm $\norm{m}$ given by the maximum displacement of the trajectory $m$, that is $\norm{m}\defeq \sup_{s\in[0,\interval]}\abs{m(s)}$, is a normed space.

On $C_0$, we consider the measure $\refmark$, given by the law of the process $X$ solution of \eqref{eq:diff:langevin} starting at $X(0)=0$. Notice that, thanks to \eqref{eq:diff:deltaDiff}, for any $\delta<\delta'/2$,
\begin{equation}\label{eq:diff:delta}
	\int_{C_0} e^{\norm{m}^{d+2\delta}}\refmark(dm)<+\infty.
\end{equation}

\begin{figure}[ht]
\begin{center}
\begin{minipage}{.9\textwidth}
\begin{minipage}{.5\textwidth}
	\includegraphics[width=\textwidth]{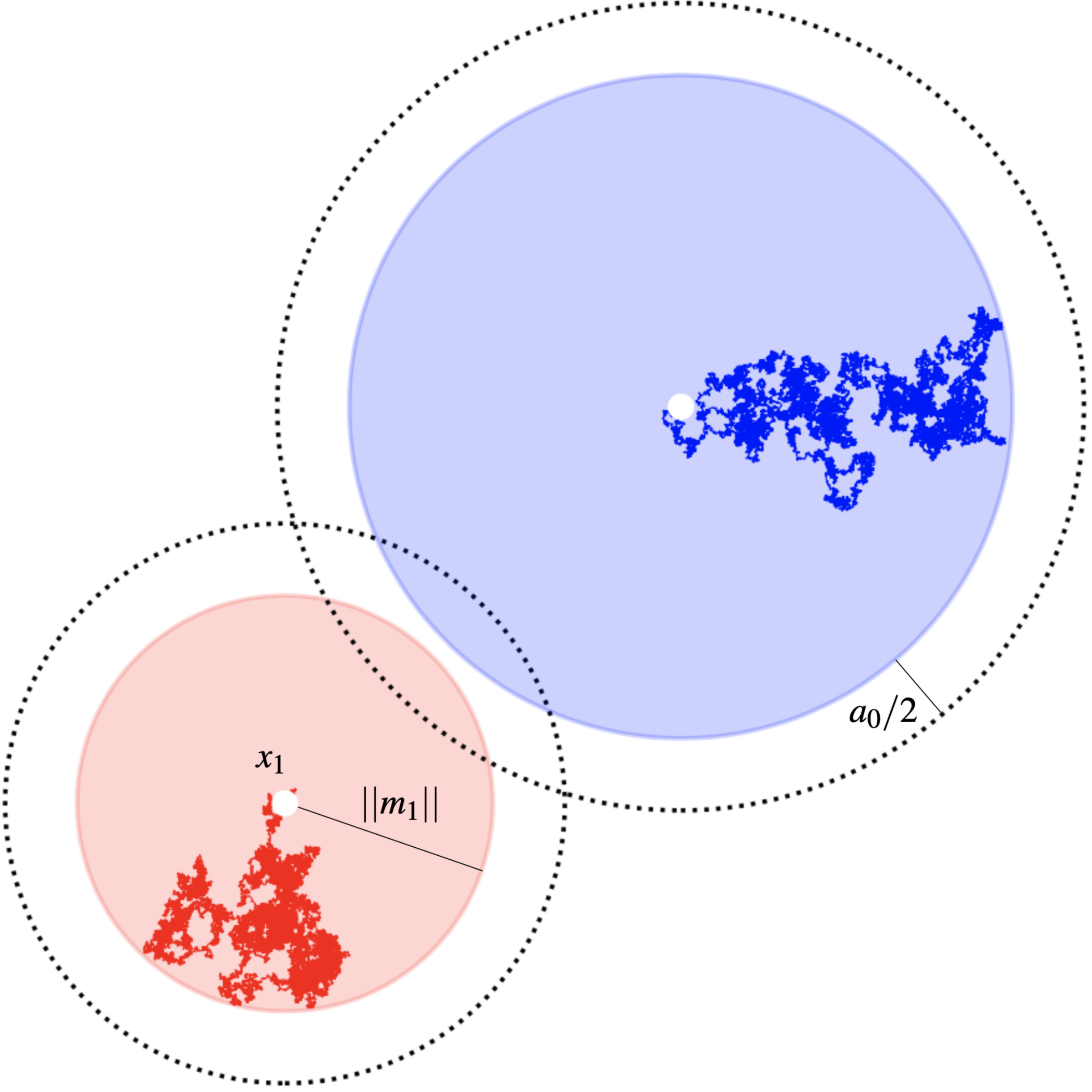}
\end{minipage}\hfill
\begin{minipage}{.4\textwidth}
	\caption{Two interacting paths of a Langevin diffusion in $\IR^2$. Each circle is centred in the starting point, while the radii of the coloured circles correspond to their maximum displacement in the time interval $[0, \interval]$; the dotted circles represent the security distance $a_0/2$ introduced in \eqref{eq:diff:range}.}
	\label{fig:diffusions}
\end{minipage}
\end{minipage}
\end{center}
\end{figure}

We can now add an interaction between two paths $\x_1=(x_1,m_1)$ and $\x_2=(x_2,m_2)$, by considering the following class of interactions, described by a path pair potential of the form
	\begin{equation*}
			\Phi(\x_1,\x_2) = \Big(\int_0^{\interval} \phi(\abs{x_1 -x_2 + m_1(s) - m_2(s)})ds\Big) \1_{[0, a_0 + \norm{m_1} + \norm{m_2}]}(\abs{x_1-x_2}).
	\end{equation*}
We provide precise definitions for the objects above in Example \ref{ex:diff:0}, after introducing the general setting in the next subsection.
	
\subsection{The general setting}\label{sec:diff:setting}
We consider here point measures on the product state space $\mathcal{E}\defeq \IR^d\times\mathcal{S}$, where $(\mathcal{S},\norm{\cdot})$ is a general normed space. On $\states$,  we take the following product measure
\begin{equation*}
	\lambda(dx,dm) = dx\otimes \refmark(dm),
\end{equation*}
where $\refmark$ is a probability measure on $\mathcal{S}$ such that \eqref{eq:diff:delta} holds.

We denote by $\confs$ the space of simple point measures (\emph{configurations}) on $\states$, i.e. of all $\sigma$-finite measures of the form 
\begin{equation*}
	\gamma=\sum_i \delta_{\x_i},\ \x_i = (x_i,m_i)\in\states,\text{ with } \x_i\neq \x_j\text{ if } i\neq j.
\end{equation*}
Since the configurations are simple, we identify them with the subset of their atoms:
\begin{equation*}
  \gamma \equiv  \big\{\x_1,\dots,\x_n,\dots\big\} \subset\mathcal{E}.
\end{equation*}
Moreover, for two disjoint configurations $\gamma,\xi\in\confs$, we denote by $\gamma\xi$ their concatenation: $\gamma\xi\defeq \gamma\cup\xi$. For $\gamma\in\confs$, $\abs{\gamma}$ denotes the number of its points, and $\confs_f\subset\confs$ is the subset of \emph{finite} configurations, i.e. with $\abs{\gamma}<+\infty$. We denote by $\underline{o}$ the configuration supported on the empty set.

For any $\Lambda\subset\IR^d$, $\confs_\Lambda\subset\confs$ denotes the subset of point measures with support in $\Lambda\times \marks$, and $\gamma_\Lambda\defeq \gamma\cap(\Lambda\times\marks)$. Let $\mathcal{B}(\IR^d)$ denote the Borel $\sigma$-algebra on $\IR^d$, and $\mathcal{B}_b(\IR^d)$ the set of bounded Borel subsets of $\IR^d$, which we often call \emph{finite volumes}. For $\Lambda\in\mathcal{B}_b(\IR^d)$, $\abs{\Lambda}$ denotes its volume.

We denote by $\mathcal{P}(\confs)$ (resp. $\mathcal{P}(\confs_\Lambda)$) the set of probability measures (or \emph{point processes}) on $\confs$ (resp. $\confs_\Lambda$). Finally, let $\IN^*\defeq\IN\setminus\{0\}$.

We consider the following measure (of infinite mass):
\begin{Definition}
	Fix $z>0$, and define the measure $\tilde{\pi}^z = 1+ \sum_{N=1}^{+\infty}\dfrac{z^N}{N!}\lambda^{\otimes N}$ on $\confs_f$. For any finite volume $\Lambda$, we consider as reference probability measure the \emph{marked Poisson point process} $\pi^z_\Lambda\in\mathcal{P}(\confs_\Lambda)$ with intensity parameter $z$, defined by normalising the restriction $\tilde\pi^z_\Lambda$ of $\tilde{\pi}^z$ to $\confs_\Lambda$ as follows:
	\begin{equation*}
		\pi^z_\Lambda(d\gamma) = e^{-z\abs{\Lambda}}\tilde\pi^z_\Lambda(d\gamma).	
	\end{equation*}
\end{Definition}
As a modification of the Poisson point process, we introduce an interaction between the marked points by considering the finite-volume Gibbs point process associated to an energy functional $\Energy$. More precisely:
\begin{Definition}
	An \emph{energy functional} $\Energy\colon\confs_f\rightarrow\IRi$
	is a measurable functional on the set of finite configurations, with $\Energy(\underline{o})=0$ by convention. In this work, we consider the energy of a finite number $N\geq 1$ of marked points to be defined, for any $\gamma = \{\x_1,\dots,\x_N\}\in\confs_f$, by the sum of a self-interaction \comment{(\emph{external field})} term and a pair-potential term:
	\begin{equation}\label{eq:diff:energy}
		\comment{\Energy}(\gamma) \defeq \sum_{i=1}^N \Psi(\x_i) + \sum_{1\leq i<j\leq N} \Phi(\x_i,\x_j) \in\IRi.
	\end{equation}
\end{Definition}

We specify later a growth condition on the self potential $\Psi$, and consider different sets of assumptions on the pair potential $\Phi\colon\states\times\states\to\IRi$. In particular, in Section \ref{sec:diff:existence} we use the notion of tempered stability (Assumption \ref{hyp:diff:0}), in Section \ref{sec:diff:RBcorrelation} we use the stronger notion of classical stability (Assumption \ref{hyp:diff:0c}), and in Section \ref{sec:diff:KS} a weak stability condition (Assumption \ref{hyp:diff:0b}). We comment on the relation between them in Remark \ref{rmk:diff:stab}.

We denote the \emph{pair-interaction} component of the energy as
\begin{equation*}
	\PairEnergy(\gamma) \defeq 	\sum_{1\leq i<j\leq N} \Phi(\x_i,\x_j),
\end{equation*}
and the \emph{conditional energy} of any marked point $\x\in\states$ given any $\xi\in\confs$ as
\begin{equation}\label{eq:diff:condEnSum}
	\CondEn{\x}{\xi} \defeq \sum_{\y\in\xi}\Phi(\x,\y).	
\end{equation}

Finally, for any $\gamma\in\confs$, let
\begin{equation*}
	\CondEn{\gamma}{\xi} \defeq \sum_{\x\in\gamma} \CondEn{\x}{\xi}.
\end{equation*}
be the conditional energy of the configuration $\gamma$ given the configuration $\xi$.

\begin{Remark}
	As the infinite sum in \eqref{eq:diff:condEnSum} is not always well defined (see Remark \ref{rmk:hyp2}), in order for the Definition \ref{def:diff:DLR} of an infinite-volume Gibbs point process $P$ to be well posed, one needs to require the conditional energy be well defined for $P$-almost all $\xi\in\confs$. We note however that this is not a restriction for Gibbs measures supported on the space of tempered configurations (see Remark \ref{rmk:hyp2}).
\end{Remark}

\begin{Definition}
	Let $\Energy$ be an energy functional as in \eqref{eq:diff:energy}. For any $\Lambda\in\mathcal{B}_b(\IR^d)$, the free-boundary-condition \emph{finite-volume Gibbs point process} on $\Lambda$ with energy functional $\Energy$, activity $z>0$ and inverse temperature $\beta > 0$ is the probability measure $P^{z,\beta}_\Lambda$ on $\confs_\Lambda$ defined by
	\begin{equation}\label{eq:diff:finiteGibbs}
		P^{z,\beta}_\Lambda(d\gamma) \defeq \dfrac{1}{Z^{z,\beta}_\Lambda}e^{-\beta \Energy(\gamma_{\Lambda})}\, \pi^z_{\Lambda}(d\gamma),
	\end{equation}
	where the \emph{partition function} $Z_\Lambda^{z,\beta}$ is the normalisation constant.
\end{Definition}

In this work we investigate the existence and uniqueness of an infinite-volume Gibbs point process, in the following sense:
\begin{Definition}\label{def:diff:DLR}
	Let $\Energy$ be an energy functional as in \eqref{eq:diff:energy}. A probability measure $P$ on $\confs$ is said to be an \emph{infinite-volume Gibbs point process} with energy functional $\Energy$, activity $z>0$ and inverse temperature $\beta > 0$, denoted $P\in\Gibbs$, if:
	
	For any $\gamma\in\confs_f$ and $P$-almost all $\xi\in\confs$, the conditional energy $\PairEnergy(\gamma_\Lambda\vert\xi_{\Lambda^c})$ of $\gamma$ given $\xi$ is well defined	, and

	$P$ satisfies, for any $\Lambda\in\mathcal{B}_b(\IR^d)$ and any positive, bounded, and measurable functional \mbox{$F:\confs\rightarrow\IR$}, the following \emph{DLR equation} (for Dobrushin--Lanford--Ruelle)
	\begin{equation}\tag{DLR}\label{eq:diff:DLR}
		\int_{\confs} F(\gamma)\, P(d\gamma) =
\int_{\confs} \dfrac{1}{Z^{z,\beta}_\Lambda(\xi)} \int_{\confs_\Lambda}F(\gamma_\Lambda\xi_{\Lambda^c}) e^{-\beta\big(\Energy(\gamma_\Lambda) +  \PairEnergy(\gamma_\Lambda\vert\xi_{\Lambda^c})\big)}\pi^z_\Lambda(d\gamma)\
P(d\xi),
	\end{equation}
where the partition function $Z_\Lambda^{z,\beta}(\xi)$ depends on the boundary condition $\xi$, the inverse temperature $\beta$, and the activity $z$.
\end{Definition}
A concept that will help in showing that such an infinite-volume measure exists is that of tempered configuration. For such a configuration $\gamma$, the number $\abs{\gamma_\Lambda}$ of its points in any finite volume $\Lambda$, should grow sublinearly w.r.t. the volume, while the norm $\norm{m}$ of its marks should grow as a fractional power of it. More precisely,
\begin{Definition}
	The set of \emph{tempered configurations} is given by the increasing union $\confs^{\tempered}\defeq \bigcup_{\tg{t}\in\IN^*}\confs^\tg{t}$, where
	\begin{equation}\label{eq:diff:temp}
		\confs^\tg{t}\defeq \big\{\gamma\in\confs: \forall l\in\IN^*,\sum_{\substack{(x,m)\in\gamma\\ \abs{x}\leq l}}(1+\norm{m}^{d+2\delta})\leq \tg{t}l^d \big\}.
	\end{equation}
	We denote by $\GibbsT\defeq \Gibbs\cap\mathcal{P}(\confs^\tempered)$ the set of \emph{tempered marked Gibbs point processes}, i.e. those whose support is included in the tempered configurations.
\end{Definition}

\section{Existence via the entropy method}\label{sec:diff:existence}

The proof of the existence of an infinite-volume Gibbs point process that we describe here makes use of the specific entropy functional as a tightness tool, as in the general approach presented in \cite{roelly_zass_2020}. In what follows we show that a tempered marked Gibbs point process
exists as soon as the interaction energy satisfies some quite natural assumptions.

In order for our model (and the example of interacting Langevin diffusions) to fit the setting of the aforementioned paper, in this section we consider energy functionals $\Energy$ that satisfy Assumptions \ref{hyp:diff:0}+\ref{hyp:diff:1}.

\setcounter{Assumption}{1}
\begin{AssumptionOpt}{}[Self-interaction growth and stability]\label{hyp:diff:0}
	\begin{description}[font = \normalfont,align=left]
		\item[]
		\item[\hyp{self}] The self-interaction potential $\Psi:\states\rightarrow\IRi$ acting on each \comment{marked point $\x=(x,m)$} is bounded from below by the opposite of a power of the norm of its mark, i.e.
	\begin{equation}\label{eq:diff:Psi}
		\exists \constPsi>0: \inf_{x\in\IR^d}\Psi(x,m)\geq -\constPsi\norm{m}^{d+\delta}.
	\end{equation}
		\item[\hyp{t.st.}] The pair potential $\Phi:\states\times\states\rightarrow\IRi$ between two \comment{marked points} is a symmetric functional that satisfies the following (mark-dependent) \emph{tempered stability} condition: there exists a constant $\stabconst_\Phi\geq 0$ such that for any finite configuration $\gamma = \{\x_1,\dots,\x_N\}\in\confs_f$,
		\begin{equation}\label{eq:diff:stab}
			\PairEnergy(\gamma) = \sum_{1\leq i<j\leq N} \Phi(\x_i,\x_j) \geq - \stabconst_\Phi \sum_{i=1}^{N}(1+\norm{m_i}^{d+\delta}).
 		\end{equation}
	\end{description}
\end{AssumptionOpt}

\begin{Assumption}[Range and conditional stability]\label{hyp:diff:1} The pair potential $\Phi:\states\times\states\rightarrow\IRi$ is such that
	\begin{description}[font = \normalfont,align=left]
 		\item[\hyp{r}] Two diffusions $\x_i,\x_j$ do not interact whenever they start \emph{too far away}: there exists a constant $a_0\geq 0$ such that 
 		\begin{equation}\label{eq:diff:range}
 			\Phi(\x_i,\x_j) = 0 \text{ whenever } \lvert x_i - x_j\rvert > a_0 + \norm{m_i} + \norm{m_j}.	
 		\end{equation}
 		\item[\hyp{t.cond.st.}] There exists a constant $\bar \stabconst_\Phi\geq 0$ such that, for any $\x=(x,m)\in\states$, for any configuration $\xi\in\confs^{\tempered}$,
 		\begin{equation}\label{eq:diff:condStab}
 			\CondEn{\x}{\xi} = \sum_{\y\in\xi}\Phi(\x,\y)\geq 	-\bar \stabconst_\Phi(1+\norm{m}^{d+\delta}).
		\end{equation}
	\end{description}
\end{Assumption}
\begin{Remark}\label{rmk:hyp2}
We briefly comment on these assumptions.
\begin{itemize}
\item The summation in \eqref{eq:diff:condStab} for the conditional energy is well defined since, as we will see in the proof of Theorem \ref{thm:diff:existence}, the range assumption \hyp{r} implies that the infinite sum of the conditional energy of $\x$ given $\xi$ is actually given by a random but finite number of terms.
\item 
\comment{The tempered stability \hyp{t.st.} is needed in order to obtain the entropy bounds presented in Subsection 3.2 of \cite{roelly_zass_2020}, and therefore the existence of an infinite-volume measure. On the other hand, the \emph{conditional} tempered stability\\ \hyp{t.cond.st.} -- which is only on \emph{tempered} boundary conditions -- is enough to recover the Gibbsianity of this object (see Subsection 3.4 of \cite{roelly_zass_2020}).
}
\end{itemize}

\end{Remark}

It is easy to show that the following Lemma holds for the support of any Gibbs point process:
\begin{Lemma}\label{lem:diff:admissible}
	For any activity $z>0$ and inverse temperature $\beta > 0$, any infinite-volume Gibbs point process $P\in\GibbsT$ is supported on configurations with locally finite energy, that is configurations $\gamma\in\confs^{\tempered}$ such that, for any $\Delta\in\mathcal{B}_b(\IR^d)$, $\PairEnergy(\gamma_\Delta)<+\infty$. Note that this is true also whenever $\Phi$ takes infinite values.
\end{Lemma}

\begin{Example}\label{ex:diff:0}
	On $\states = \IR^d\times C_0$, consider the following class of interactions, described by a path pair potential of the form
	\begin{equation}\label{eq:diff:2body}
			\Phi(\x_i,\x_j) = \Big(\int_0^{\interval} \phi(\abs{x_i -x_j + m_i(s) - m_j(s)})ds\Big) \1_{[0, a_0 + \norm{m_i} + \norm{m_j}]}(\abs{x_i-x_j}),
	\end{equation}
	with $\phi$ given by the sum of two potentials on $\IR_+$: $\phi = \phi_{hc} + \phi_{l}$, where
	\begin{itemize}
			\item The potential $\phi_{hc}$ is pure \emph{hard core} at some diameter $R>0$, that is
			\begin{equation*}
				\phi_{hc}(u)=(+\infty)\1_{[0,R)}(u).	
			\end{equation*}
			\item The potential $\phi_l$ satisfies a stability property, i.e. there exists a constant $\stabconst_\phi\geq 0$ such that, for any admissible configuration $\{y_1,\dots,y_N\}$, $N\geq 1$, the following holds (see \cite{ruelle_1969}, paragraph 3.2.5):
		\begin{equation}\label{eq:diff:hcstab}
			\sum_{i=1}^N \phi_l(\abs{y_i})\geq -2\stabconst_\phi,
		\end{equation}
		where a finite configuration $\{y_1,\dots,y_N\}\subset\IR^d$, $N\geq 1$, is called \emph{admissible} if, for any pair $y_i\neq y_j$, $\phi(\abs{y_i-y_j})<+\infty$. 
		\end{itemize}
		Note that \eqref{eq:diff:hcstab} implies that the potential $\phi$ -- defined on the location space $\IR^d$ -- is \emph{stable} in the sense of Ruelle (see \cite{ruelle_1970}), with stability constant $\stabconst_\phi$, i.e.
		\begin{equation*}
			\forall N\geq 1,\, \forall\{y_1,\dots,y_N\}\subset\IR^d,\quad \sum_{1\leq i<j\leq N} \phi(\abs{y_i-y_j})\geq - \stabconst_\phi N.
		\end{equation*}
		Note also how the coefficient $a_0$ here plays the role of a \emph{sensitivity parameter} (see Figure \ref{fig:diffusions}): if the pair potential $\phi$ is repulsive (i.e. positive), then $a_0$ can take any finite positive value. If instead on some region $\phi$ is attractive (i.e. negative), \comment{in order for $\Phi$ to inherit the stability property from $\phi$,} $a_0$ should be chosen in such a way that $\phi$ remains attractive on $[a_0,+\infty)$: $\phi(u)\leq 0$ if $u\geq a_0$ (see Figure \ref{fig:potential}, and \comment{cf. the notion of \emph{tempered potential} of \cite{ruelle_1969}}). Indeed, it is easy to see that, in this case, \eqref{eq:diff:stab} holds with $\stabconst_\Phi = \stabconst_\phi$.
		
		Finally, note that, if $\phi$ is a \emph{finite-range} pair potential, i.e. supported on $[0,r]$, for $r>0$, then 
		\begin{equation*}
			\int_0^{\interval} \phi(\abs{x_i -x_j + m_i(s) - m_j(s)})ds = \Big(\int_0^{\interval} \phi(\abs{x_i -x_j + m_i(s) - m_j(s)})ds\Big) \1_{[0, r + \norm{m_i} + \norm{m_j}]}(\abs{x_i-x_j}).
		\end{equation*}
		
		We now show that this class of potentials satisfies Assumption \ref{hyp:diff:1}.
		\begin{proof}
		Firstly, thank to the previous Lemma, we can actually restrict our study to the admissible configurations. Setting $\tg{l}(\tg{t})\defeq 2^{\frac{d+\delta}{\delta} - 1}\tg{t}^{\frac{1}{\delta}}$, one can see that the range of the interaction is bounded by
		\begin{equation*}
			\tg{r}(\gamma,\Lambda) = 3\tg{d}(\Lambda) + 2\tg{l}(\tg{t}) + 2\sup_{(x,m)\in\gamma_\Lambda}\norm{m} + 1 + a_0,
		\end{equation*}
		where $\tg{d}(\Lambda)$ is the smallest $R>0$ such that $\Lambda\subset B(0,R)$, i.e. for any $\x=(x,m)\in\states$ and $\xi\in\confs^{\tg{t}}$, $\tg{t}\geq 1$, setting $\Delta \defeq B(x,\tg{r}(\gamma,\Lambda))$, the conditional energy $\CondEn{\x}{\xi\setminus\{x\}}$ of $\x$ given $\xi$ is actually given by $\CondEn{\x}{\xi_{\Delta\setminus\{x\}}}$.
	In order to obtain the lower bound for the conditional energy in this context, it is unfortunately not true -- as used instead in Section 4 of \cite{roelly_zass_2020} -- that we can control the number of points of $\xi_\Delta$, \emph{uniformly in $\x$}. On the other hand, thanks to Lemma \ref{lem:diff:admissible}, we can assume that $\xi_\Delta$ is of finite energy, and therefore use \eqref{eq:diff:hcstab} to estimate
	\begin{equation*}
	\begin{split}
		\sum_{\x_j\in\xi_{\Delta\setminus\Lambda}} \Phi(\x_i,\x_j) &= \int_0^{\interval} \sum_{\x_j\in\xi_{\Delta\setminus\Lambda}} \phi(\abs{x_i-x_j+m_i(s)-m_j(s)})ds\, \1_{\{\lvert x_i - x_j\rvert\leq a_0 + \norm{m_i} + \norm{m_j}\}}\\
		&\overset{\mathclap{\eqref{eq:diff:hcstab}}}{\geq}\  \int_0^{\interval}  -2\stabconst_\phi \geq -2\stabconst_\phi. 
	\end{split}
	\end{equation*}
	Together with the stability of $\gamma_\Lambda\mapsto \Energy(\gamma_\Lambda)$, this yields the following lower bound for the conditional energy:
	\begin{equation*}
		H_\Lambda(\gamma_\Lambda\xi_{\Lambda^c})\geq -\beta(A_\Psi\vee 2\stabconst_\phi)\sum_{\x\in\gamma_\Lambda}(1+\norm{m}^{d+\delta}).\qedhere
	\end{equation*}
	\end{proof}
	Notice how, under these conditions, the trajectories of two interacting paths $\x_1 = (x_1,m_1)$ and $\x_2 = (x_2,m_2)$ are allowed to intersect, but at each time $s$ the paths keep at a distance of at least $R$; the hard-core component, indeed, imposes $\abs{x_1+m_1(s) - x_2+m_2(s)}\geq R$ for any $s\in[0,1]$.
	
	\medskip
	\noindent\emph{A particular case:} Let $\phi$ be given by the sum of a hard-core component and a shifted Lennard--Jones potential, i.e.
		\begin{equation*}
			\phi(u) = \phi_{hc}(u) + \phi_{LJ}(u-R)\1_{[R,+\infty)}(u),\ u\in\IR_+,
		\end{equation*}
		where $\phi_{LJ}(u) = \frac{a}{u^{12}} - \frac{b}{u^6}$, $a,b>0$. Pictured in Figure \ref{fig:potential} is an example with $R = 1$. We remark that this potential has a non-integrable growth in a neighbourhood of its hard-core component; in particular, it does not satisfy Assumption \ref{hyp:diff:2} below, which is used for the uniqueness proof.
	\begin{figure}[t]
	\begin{center}
	\begin{minipage}{.9\textwidth}
	\begin{minipage}[b]{.5\textwidth}
		\includegraphics[width=1.2\textwidth]{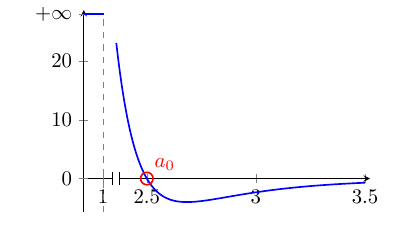}
	\end{minipage}\hfill
	\begin{minipage}[b]{.4\textwidth}
		\caption{A shifted Lennard--Jones potential $\phi_{LJ}(u-1) = 16\big(\big(\frac{3/2}{u-1}\big)^{12} - \big(\frac{3/2}{u-1}\big)^{6}\big)$ with hard-core diameter $R = 1$; it is always negative after $a_0=2.5$, and explodes as $x\to 1^+$.}
		\label{fig:potential}
	\end{minipage}
	\end{minipage}
	\end{center}
	\end{figure}	
\end{Example}

\begin{Example}
	One can consider a class of \emph{translation-invariant} pair potentials. More precisely, let $\Phi$ be invariant by translation: $\Phi(\x_i,\x_j) = \Phi(\x_i-\x_j)$, with
		\begin{equation*}
			\Phi(\x) = \Big(\int_0^{\interval} \phi(\abs{x + m(s)})ds\Big) \1_{\{\lvert x\rvert\leq a_0 + \norm{m}\}},
		\end{equation*}
		where $\phi$ is given by the above sum of a hard-core component and a shifted Lennard--Jones potential.
\end{Example} 

We can now state our existence result.
\begin{Theorem}\label{thm:diff:existence}
	Let $\Energy$ be an energy functional as in \eqref{eq:diff:energy}, satisfying Assumptions \ref{hyp:diff:0}+\ref{hyp:diff:1}. For any $z>0$ and $\beta > 0$, there exists at least one infinite-volume tempered Gibbs point process $P^{z,\beta}\in\GibbsT$. 
\end{Theorem}
\begin{proof}
	Let $z>0$, $\beta > 0$. In order to apply the existence result of \cite{roelly_zass_2020} to this 
	context, we show that a stability condition holds both for the energy of a finite configuration and for the conditional energy, and that the random interaction range is finite (possibly unbounded). These conditions are called in \cite{roelly_zass_2020}, \hyp{st.}, \hyp{r}, and \hyp{loc.st.}.
	\begin{description}[leftmargin=0pt,font=\sffamily]
		\item[Step 1.] The conditions \eqref{eq:diff:Psi} and \eqref{eq:diff:stab} -- on the self interaction and pair potential, respectively, yield the following \emph{stability} for the energy of a finite number of paths:
		\begin{equation*}
			\forall \gamma\in\confs_f,\quad \Energy(\gamma)\geq -(\constPsi\vee\stabconst_\Phi)\big( \sum_{(x,m)\in\gamma}(1+\norm{m}^{d+\delta}\big).
		\end{equation*}
		\item[Step 2.] We now focus on analysing the \emph{range of the interaction}: we show that for any tempered configuration $\gamma\in\confs^{\tg{t}}$, $\tg{t}\geq 1$, and for any finite volume $\Lambda$, there exists a positive number $\tg{r}= \tg{r}(\gamma,\Lambda)$ such that
		\begin{equation}\label{eq:diff:condEn1}
			\CondEn{\x}{\xi} = \sum_{\substack{\y\in\xi\\ 0<\abs{y-x}\leq \tg{r}}}\Phi(\x,\y).
		\end{equation}
	
		Set $\tg{l}(\tg{t})\defeq 2^{\frac{d+\delta}{\delta} - 1}\tg{t}^{\frac{1}{\delta}}$. Using the definition of tempered configurations, one has that, for all $l\geq \tg{l}(\tg{t})$ and for any $\x\in\gamma\in\confs^{\tg{t}}$ such that $\abs{x}>2l+1+a_0$,
		\begin{equation*}
			\abs{x}-\norm{m}\overset{\eqref{eq:diff:temp}}{\geq} \abs{x} - \frac{1}{2}\lceil\abs{x}\rceil	\geq l+a_0.
		\end{equation*}
		Thanks to condition \hyp{r}, this means that the range of the interaction is bounded by
		\begin{equation*}
			\tg{r}(\gamma,\Lambda) = 3\tg{d}(\Lambda) + 2\tg{l}(\tg{t}) + 2\sup_{(x,m)\in\gamma_\Lambda}\norm{m} + 1 + a_0,
		\end{equation*}
		where $\tg{d}(\Lambda)$ is the smallest $R>0$ such that $\Lambda\subset B(0,R)$.
		\item[Step 3.] Fix $\Lambda\in\mathcal{B}_b(\IR^d)$, and consider, for $\gamma\in\confs$ and $\xi\in\confs^{\tempered}$, the conditional energy of $\gamma_\Lambda$ given $\xi_{\Lambda^c}$, that is:
		\begin{equation*}
			H_\Lambda(\gamma_\Lambda\xi_{\Lambda^c}) \defeq \Energy(\gamma_\Lambda) + \CondEn{\gamma_\Lambda}{\xi_{\Lambda^c}}.
    	\end{equation*}
    	Thanks to \eqref{eq:diff:condEn1}, denoting $\Delta\defeq \Lambda\oplus B(0,\tg{r}(\gamma,\Lambda))$, we have
	\begin{equation}\label{eq:diff:condEn2}
		\CondEn{\gamma_\Lambda}{\xi_{\Lambda^c}} = \CondEn{\gamma_\Lambda}{\xi_{\Delta\setminus\Lambda}} = \sum_{\x_i\in\gamma_\Lambda}\sum_{\x_j\in\xi_{\Delta\setminus\Lambda}} \Phi(\x_i,\x_j).
	\end{equation}
	From \hyp{t.cond.st.}, we have that $\CondEn{\gamma_\Lambda}{\xi_{\Lambda^c}}\geq - \sum_{\x_i\in\gamma_\Lambda}\bar \stabconst_\Phi(1+\norm{m}^{d+\delta})$. Together with the stability of $\gamma_\Lambda\mapsto \Energy(\gamma_\Lambda)$, this yields the following lower bound for the conditional energy:
	\begin{equation*}
		H_\Lambda(\gamma_\Lambda\xi_{\Lambda^c})\geq -(A_\Psi\vee\stabconst_\Phi\vee\bar \stabconst_\Phi)\sum_{\x\in\gamma_\Lambda}(1+\norm{m}^{d+\delta}),
	\end{equation*}
	
	\end{description}

	We can now apply, having checked its three conditions \hyp{st.}, \hyp{r}, and \hyp{loc.st.}, Theorem 1 of \cite{roelly_zass_2020}: there exists an infinite-volume Gibbs measure $P^{z,\beta}\in\GibbsT$. 
	
\end{proof}


We also obtain a Ruelle bound for the correlation functions of any tempered Gibbs point process.

\begin{Definition}
	Consider a configuration $\gamma\in\confs$. For any $N\geq 1$, its \emph{factorial measure of order $N$} is given by 
	\begin{equation*}
		\gamma^{(N)}(d\x_1,\dots,d\x_N) \defeq \gamma(d\x_1)(\gamma\setminus\{\x_1\})(d\x_2)\dots(\gamma\setminus\{\x_1,\dots,\x_{N-1}\})(d\x_N).
	\end{equation*}
	By taking the expectation (if it exists) under a point process $P$, we obtain its \emph{$N$-th factorial moment measure}: a measure $\alpha_N^{(P)}$ on $\states^N$ defined by
	\begin{equation*}
		\alpha_N^{(P)}(\cdot) \defeq \IE_P[\gamma^{(N)}(\cdot)].
	\end{equation*}
\end{Definition}
One can then consider, for any $N\geq 1$, its $N$-point correlation function, defined as the Radon--Nikodym derivative of its $N$-th factorial moment measure $\alpha_N^{(P)}$ with respect to the product measure $(z\sigmab)^{\otimes N}$, where
\begin{equation}\label{eq:diff:refSelf}
	\sigmab(d\x) \defeq e^{-\beta\, \Psi(\x)}\lambda(d\x).		
\end{equation}

\begin{Proposition}[\cite{nguyen_zessin_1979}]\label{prop:Ncorr}
	Let $P\in\Gibbs$, $z>0$, $\beta > 0$. Its \emph{$N$-point correlation function} admits, for $\sigmab^{\otimes N}$-almost all $(\x_1,\dots,\x_N)\in\states^N$, the following representation:
	\begin{equation}\label{eq:diff:corrGibbs}
		\corr{N}{P}(\x_1,\dots,\x_N) = e^{-\beta \PairEnergy(\x_1,\dots,\x_N)} \int_\confs e^{-\beta \CondEn{\x_1,\dots,\x_N}{\xi}}P(d\xi).
	\end{equation}
\end{Proposition}
\begin{Remark}
	Note that $\corr{N}{P}(\cdot)$ is a symmetric function, as for any $(\x_1,\dots,\x_N)\in\states^k$ and any permutation $\{i_1,\dots,i_N\}$, $\corr{N}{P}(\x_{i_1},\dots,\x_{i_N}) = \corr{N}{P}(\x_1,\dots,\x_N)$. 
\end{Remark}

\begin{Proposition}\label{prop:diff:existenceRB}
	For any $N\geq 1$, the $N$-point correlation function $\corr{N}{P^{z,\beta}}$ of any Gibbs point process $P^{z,\beta}\in\GibbsT$ satisfy a \emph{Ruelle bound}: for $\sigmab^{\otimes N}$-almost all $(\x_1,\dots,\x_N)\in\states^N$,
	\begin{equation}
		\corr{N}{P^{z,\beta}}(\x_1,\dots,\x_N)\leq \prod_{i=1}^N \constan(\x_i),
	\end{equation}
	where $\constan(x,m) \defeq \exp\left(\beta(\stabconst_\Phi+\bar\stabconst_\Phi)(1+\norm{m}^{d+\delta})\right).$
	\end{Proposition}
	\begin{proof}
		Putting together \eqref{eq:diff:stab} and \eqref{eq:diff:condStab}, we estimate
	\begin{equation*}
		\corr{N}{P^{z,\beta}}(\x_1,\dots,\x_N)\leq e^{\beta\sum_{i=1}^N\stabconst_\Phi (1+\norm{m_i}^{d+\delta})} \int_{\confs^\tempered} e^{\beta\sum_{i=1}^N\bar \stabconst_\Phi(1+\norm{m_i}^{d+\delta})}P^{z,\beta}(d\xi),
	\end{equation*}
	yielding the desired bound.	
\end{proof}

\setcounter{Example}{1}
\begin{Excont}[continued]
	For the class of potentials described in Example \ref{ex:diff:0}, the Ruelle bound holds uniformly in $(\x_1,\dots,\x_N)\in\states^N$, and is of the form \begin{equation}\label{eq:diff:exunifRB}
 	\corr{N}{P^{z,\beta}}(\x_1,\dots,\x_N)\leq e^{3\beta\stabconst_\phi N}.
\end{equation}
\end{Excont}

\begin{Remark}\label{rmk:KS}
\comment{One could use the Kirkwood--Salsburg approach, presented in Section \ref{sec:diff:KS} in the context of uniqueness, to prove the existence of an infinite-volume Gibbs point process as well, in some activity regime. 

However, while Kirkwood--Salsburg yields a limiting correlation function, one still needs to prove that it is indeed the correlation function of a Gibbs point process.
In order to do this, one needs additional regularity conditions (other than Assumption \ref{hyp:diff:4} below) on the potential. In Theorem 3 of \cite{poghosyan_zessin_2020}, for example, the authors propose sufficient assumptions (like \emph{modified regularity}) for existence which are -- to the best of our knowledge -- as weak as possible for the method. 

For this reason, we used here the entropy method (which yields an existence result for any $z>0$) to improve the quality of the result.}
\end{Remark}

\section{Ruelle bounds for correlation functions}\label{sec:diff:RBcorrelation}

At the end of the previous section, we have seen how a Ruelle bound for the correlation functions of a Gibbs point process holds. This is an essential estimate when proving uniqueness of such a process, as we will see in Section \ref{sec:diff:KS}.

Suppose now that you have a pair potential $\Phi$ -- not necessarily satisfying the assumptions of the previous section -- and that you already have an infinite-volume Gibbs point process $P\in\Gibbs$, not necessarily constructed as above. Using tools from cluster expansion (see, for example, \cite{ruelle_1963, ruelle_1970}) we find, for any inverse temperature $\beta>0$, a domain of activity $(0,\zRuelle)$ such that, for any $z \in(0,\zRuelle)$, the correlation functions of $P$ satisfy a Ruelle bound. An important tool is given by the Ursell kernel (see the work by R.A. Minlos and S. Poghosyan in \cite{minlos_poghosyan_1977}), introduced here in Subsection \ref{sec:diff:TreeEstimates}.

In this section we work under Assumptions \ref{hyp:diff:0c}+\ref{hyp:diff:2}, i.e. the classical stability assumption \eqref{eq:diff:stabunif} and an additional regularity condition, introduced in the next subsection. 

\subsection{Correlation functions}

While we have so far decomposed the energy functional in \eqref{eq:diff:energy} into self- and pair-interaction terms, in order to set ourselves in the framework of cluster expansion -- that typically deals exclusively with pair interactions -- in what follows we include the self-interaction term in the reference measure and define, for $z>0$, the measure
\begin{equation*}
	\tilde\pi^{z\sigmab} = \sum_{N=0}^{+\infty} \frac{z^N}{N!}\sigmab^{\otimes N},
\end{equation*}
and the corresponding Poisson point process $\pi^{z\sigmab}$, where $\sigmab$ is as in \eqref{eq:diff:refSelf}. The finite-volume Gibbs point process $P^{z,\beta}_\Lambda$ defined in \eqref{eq:diff:finiteGibbs} on $\confs_\Lambda$ can then be equivalently defined using $\pi^{z\sigmab}$ and just the pair interaction $\PairEnergy(\gamma) = \sum_{\{\x,\y\}\subset\gamma}\Phi(\x,\y)$ (in place of the full energy functional $\Energy$):
\begin{equation*}
	P^{z,\beta}_\Lambda(d\gamma) = \dfrac{1}{Z^{z\sigmab,\beta}_\Lambda}e^{-\beta\PairEnergy(\gamma_{\Lambda})}\, \pi^{z\sigmab}_{\Lambda}(d\gamma),
\end{equation*}
where $Z_\Lambda^{z\sigmab,\beta}$ is the normalisation constant.

As we have already mentioned, the proof of the uniqueness of the Gibbs point process revolves around the study of its correlation functions, which we now introduce.
We start by introducing a finite-volume correlation function induced by the interaction $\Phi$:
\begin{Definition}
	Let $\beta > 0$, $z>0$. For any finite volume $\Lambda\subset\IR^d$, the \emph{finite-volume correlation function} $\corr{\Lambda}{z,\beta}$ in $\Lambda$ (with free boundary condition) is given, for any $\gamma\in\confs_\Lambda$, by
	\begin{equation*}
		\corr{\Lambda}{z,\beta}(\gamma) = \dfrac{1}{\tilde Z^{z\sigmab,\beta}_\Lambda}\int_{\confs_\Lambda} e^{-\beta\PairEnergy(\xi\gamma)}\tilde \pi_\Lambda^{z\sigmab}(d\xi),
	\end{equation*}
	where $\tilde{Z}^{z\sigmab,\beta}_\Lambda$ is the normalisation constant.
\end{Definition}

\setcounter{Assumption}{1}
\begin{AssumptionOpt}{$^*$}[Classical stability]\label{hyp:diff:0c}
	Consider an energy functional $\Energy$ as in \eqref{eq:diff:energy}, where the self-potential $\Psi$ satisfies \eqref{eq:diff:Psi}, but for which the stability condition \hyp{t.st.} of the pair potential $\Phi$ is replaced by a stronger one:
\begin{description}[font = \normalfont,align=left]
	\item[\hyp{st.}] There exists a constant $\stabconst_\Phi\geq 0$ such that for any finite configuration $\gamma = \{\x_1,\dots,\x_N\}\in\confs_f$,
	\begin{equation}\label{eq:diff:stabunif}
		\PairEnergy(\gamma) = \sum_{1\leq i<j\leq N} \Phi(\x_i,\x_j) \geq - \stabconst_\Phi N.
	\end{equation}
\end{description}	
\end{AssumptionOpt}

\begin{Remark}
	Clearly, \hyp{st.} implies \hyp{t.st.}. Moreover, note that, from the stability 
	of the pair potential $\Phi$, there exists a functional
	\begin{equation*}
		\msb{i}:\confs_f\setminus\{\underline{o}\}\rightarrow\states		
	\end{equation*}
	such that for any non-empty configuration $\gamma$ there exists $\msb{i}(\gamma)\in\gamma$ where the sum of its interactions with the other marked points in $\gamma$ is bounded from below: 
	\begin{equation}\label{eq:diff:I}
		\forall \gamma\in \confs_f\setminus\{\underline{o}\},\quad \CondEn{\msb{i}(\gamma)}{\gamma\setminus\{\msb{i}(\gamma)\}} \geq -2\stabconst_{\Phi}.
	\end{equation}
	As a consequence, $\Phi$ is bounded from below by $-2\stabconst_\Phi$:
	\begin{equation}\label{eq:diff:bounded}
		\inf \Phi(\x,\y) \geq -2\stabconst_\Phi.
	\end{equation}
	In Example \ref{ex:diff:potential2} below we make use of \eqref{eq:diff:bounded}, while \eqref{eq:diff:I} is used in Proposition \ref{prop:diff:Q}. 
\end{Remark}

We further assume that:
\addtocounter{Assumption}{2}
\begin{AssumptionOpt}{}[regularity]\label{hyp:diff:2}
	The pair potential $\Phi$ satisfies the following uniform \emph{regularity condition}:
	\begin{equation*}
		\RuelleC\defeq \sup_{\x\in\states}\int_\states\abs{e^{-\beta\Phi(\x,\y)}-1}\sigmab(d\y) < +\infty.
	\end{equation*}
\end{AssumptionOpt}

\setcounter{Example}{1}
\begin{Excont}[continued]\label{ex:diff:potential2}
	If the potential $\phi = \phi_{hc}+\phi_l$ is \emph{integrable} outside of the hard core, that is 
	\begin{equation*}
		\norm{\phi}_{R_+}\defeq \int_R^{+\infty} \abs{\phi_l(u)}u^{d-1}du<+\infty,
	\end{equation*}
	then Assumption \ref{hyp:diff:2} holds. Indeed, since for any $x\in(-\infty,+\infty]$, 
	\begin{equation}\label{eq:diff:basicbound}
		\abs{e^{-x}-1}\leq x^-e^{x^-} + (1-e^{-x^+})\leq \abs{x} e^{x^-},		
	\end{equation}
	where $x^-\defeq \max(0,-x)$ and $x^+ \defeq \max(0,x)$ are the negative and positive part of x, respectively, we have
	\begin{equation*}
		\lvert e^{-\beta\Phi}-1\rvert \leq \beta\abs{\bar \Phi}e^{2\beta\stabconst_\Phi},
	\end{equation*}
	where we denote by $\bar\Phi$ the \emph{truncated pair potential}, defined, for any $\x,\y\in\states$, by
	\begin{equation*}
		\bar\Phi(\x,\y) = \begin{cases}
 			1 &\text{ if } \Phi(\x,\y)=+\infty\\
 			\Phi(\x,\y) &\text{ otherwise}.
 		\end{cases}
	\end{equation*}
	Let $\bar\phi(u)\defeq \1_{\{u < R\}} + \phi(u)\1_{\{u\geq R\}}$.
	Using the above bound, we can estimate, for any $\x_1\in\states$, 
	\begin{equation*}
	\begin{split}
		&\int_\states\abs{e^{-\beta\Phi(\x_1,\x_2)}-1}\sigmab(d\x_2)\leq e^{2\beta\stabconst_\phi}\int_\states\beta\abs{\bar\Phi(\x_1,\x_2)}\, \sigmab(d\x_2)\\
		&\leq e^{2\beta\stabconst_\Phi}\int_\states \int_0^{\interval} \beta\abs{\bar\phi(x_2 + m_2(s) - x_1 - m_1(s))}ds\ \1_{\{\lvert x_2-x_1\rvert\leq a_0 + \norm{m_2} + \norm{m_1}\}} \, \sigmab(d\x_2)\\
		&\overset{\eqref{eq:diff:Psi}}{\leq} e^{2\beta\stabconst_\Phi}\int_{C_0}\int_0^{\interval}\int_{\IR^d} \beta\abs{\bar\phi(x_2 + m_2(s) - x_1 - m_1(s))}\, dx_2\, ds\, e^{\constPsi\norm{m_2}^{d+\delta}}\refmark(dm_2)\\
		&\leq e^{2\beta\stabconst_\Phi}\beta \left(b_d R^d + \norm{\phi}_{R_+}\right) \int_{C_0}e^{\constPsi\norm{m_2}^{d+\delta}}\refmark(dm_2),
	\end{split}
	\end{equation*}
	which is finite thanks to the ultra-contractivity assumption, see \eqref{eq:diff:delta}.
	
	\medskip
	\noindent\emph{A particular case:} Suppose the potential $\phi$ is given by the sum of a hard-core potential $\phi_{hc}$ in $[0,R)$ and the Lennard--Jones potential $\phi_l\equiv\phi_{LJ}$ in $[R,+\infty)$. In particular, it is finite in $[R,+\infty)$, with maximum $\phi_l(R)$. Pictured in Figure \ref{fig:Penrose} is an example with $R=1$.
	\begin{figure}[t]
	\begin{center}
	\begin{minipage}{.9\textwidth}
	\begin{minipage}[b]{.5\textwidth}
		\includegraphics[width=1.2\textwidth]{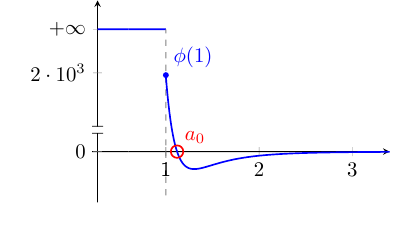}
	\end{minipage}\hfill
	\begin{minipage}[b]{.4\textwidth}
		\caption{The sum of a hard-core potential $\phi_{hc}$ and the Lennard--Jones potential $\phi_{LJ}$. The integrable component $\phi_{LJ}$ of the potential has a maximum in $1$.}
		\label{fig:Penrose}
	\end{minipage}
	\end{minipage}
	\end{center}
	\end{figure}
	\end{Excont}
	
\subsection{Cluster expansion: Ursell kernel and tree-graph estimates}\label{sec:diff:TreeEstimates}

In this subsection, after introducing the \emph{Ursell kernel}, we use it to rewrite the correlation functions of a Gibbs point process and -- following an approach inspired by \cite{kks_1998,kuna_1999} -- use \emph{tree-graph estimates} to obtain a Ruelle bound for them. Our innovation comes from being able to obtain that the correlation functions of any Gibbs point process satisfy a Ruelle bound with the same constant $\constanz$, uniformly in the finite volume, therefore yielding uniqueness in the set of tempered Gibbs point processes.

We consider here \emph{undirected connected graphs}. For any non-empty set $V\subset\IR^d$, a \emph{graph} $G\equiv G(V)$ on $V$ is given by the pair $(V,E)$, where $V$ is the vertex set, and the set of edges $E$ is a subset of $\{\{x,y\}\subset V:x\neq y\}$. Indeed, we write $\{x,y\}\in G$ to denote the edge $xy\in E$ between two vertices $x,y\in V$. A \emph{tree} $T$ is a connected graph without loops. We also introduce the following notations:
\begin{itemize}
	\item $\mathscr{C}_n(V)$ denotes the set of all undirected connected graphs with $n$ vertices belonging to $V$. 
	\item $\mathscr{T}(V)$ denotes the set of all \emph{trees} on $V$.
\end{itemize}
Note that the notion of graph $G=(V,E)\in\mathscr{C}_n(V)$ does not depend on the possible orderings of the points of the vertex set $V=\{x_1,\dots,x_n\}\subset\IR^d$. Moreover, when there is no risk of confusion, we identify a graph $G$ on $\{x_1,\dots,x_n\}$ with the corresponding one on the index set $\{1,\dots,n\}\in\IN$ (i.e. where the edge $\{x_i,x_j\}$ corresponds with the edge $\{i,j\}$, see Figure \ref{fig:diff:Tree}).

When using these notations on a finite configuration $\gamma\subset\IR^d\times \marks$, with an abuse of notations, we write $\mathscr{C}_n(\gamma)$ as shorthand for $\mathscr{C}_n(\text{proj}_{\IR^d}(\gamma))$ (analogously for $\mathscr{T}$). 
\begin{figure}[ht]
\begin{center}
\begin{minipage}{.9\textwidth}
	\begin{minipage}[]{.3\textwidth}
		\includegraphics{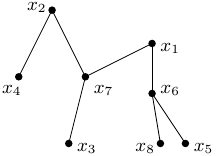}
	\end{minipage}\hfill
	\begin{minipage}[]{.6\textwidth}
		\caption{Example of a treee $T\in\mathscr{T}(V)$, where $V = \{x_1,\dots,x_8\}\subset\IR^2$. It can be equivalently described by placing the points of $V$ on the vertices of a tree $\tilde T$ on $\{1,\dots,8\}\in\IN$. More precisely, $\tilde T$ on $\{1,\dots,8\}$ is constructed by placing an edge $\{i,j\}\in \tilde T$ if and only if there is an edge $\{x_i,x_j\}\in T$.}
		\label{fig:diff:Tree}
	\end{minipage}
\end{minipage}
\end{center}
\end{figure}	

\begin{Definition}
	For any two measurable functionals $F,G\colon\confs_f\rightarrow\IR$, define their \emph{$*$-product} by
\begin{equation*}
	(F*G)(\gamma)\defeq \sum_{\xi\subset\gamma} F(\gamma\setminus\xi)G(\xi), \quad \gamma\in\confs_f.
\end{equation*}
with identity $1^*(\gamma) \defeq \1_{\{\gamma=\underline{o}\}}$. The space of measurable functionals with this operation is an algebra $\mathcal{A}$. Moreover, the set
\begin{equation*}
	\mathcal{A}_0\defeq \{ F\in\mathcal{A}\colon F(\underline{o}) = 0\}	
\end{equation*}
is an ideal of $\mathcal{A}$.
The \emph{exponential} and \emph{logarithm operators} are defined by
\begin{equation*}
	\exp^*F \defeq \sum_{n\geq 0}\tfrac{1}{n!}F^{*n},\quad \log^*(1^*+F)\defeq \sum_{n\geq 1}\tfrac{(-1)^{n-1}}{n} F^{*n}.
\end{equation*}
\end{Definition}

\begin{Definition}[Ursell function and kernel]
We introduce the two following notions:
\begin{itemize}
	\item The \emph{Ursell function} $k\colon\confs_f\rightarrow\IR$ is a functional on finite configurations, defined by setting
	\begin{equation*}
		k(\gamma) \defeq \log^*(e^{-\beta\PairEnergy})(\gamma), \quad \gamma\in\confs_f.
	\end{equation*}
	Equivalently (\cite{kks_1998}, Proposition 4.3), $k(\underline{o}) = 0$ and, for any $\gamma$ with $\abs{\gamma}=n\geq 1$,
	\begin{equation*}
		k(\gamma) = \sum_{G\in\mathscr{C}_n(\gamma)}\prod_{\{\x,\y\}\in G} \left( e^{-\beta\Phi(\x,\y)} - 1\right).
	\end{equation*}
	\item The \emph{Ursell kernel} $\bar{k}\colon\confs_f\times\confs_f\rightarrow\IR$ is defined on disjoint configurations by
	\begin{equation*}
		\bar k(\gamma,\xi) \defeq \left[\exp^*(-k)*(e^{-\beta\PairEnergy})(\gamma,\cdot)\right](\xi), \quad\gamma,\xi\in\confs_f,\ \gamma\cap\xi=\underline{o}.
	\end{equation*}
\end{itemize}
\end{Definition}
The Ursell kernel relates to the Ursell function as follows:
\begin{Lemma}[\cite{kks_1998}, Lemma 4.6]\label{lem:diff:KbarK}
	For any finite configuration $\gamma\neq\underline{o}$,
	\begin{equation*}
		\forall \x\in\gamma,\quad \bar k(\{\x\},\gamma\setminus\{\x\}) = k(\gamma).	
	\end{equation*}
\end{Lemma}
Moreover, it provides a new expression for the correlation functions:
\begin{Lemma}[\cite{kks_1998}, Proposition 4.5]\label{lem:diff:barKrho}
	Let $\gamma\in\confs_\Lambda$, $\Lambda\in\mathcal{B}_b(\IR^d)$. If $\int_{\confs_\Lambda}\abs{k(\xi)}\tilde\pi^{z\sigmab}(d\xi)<+\infty$, then
	\begin{equation}\label{eq:diff:corrBark}
		\corr{\Lambda}{z,\beta}(\gamma) = 	\int_{\confs_\Lambda} \bar k(\gamma,\xi)\tilde\pi^{z\sigmab}(d\xi).
	\end{equation}
\end{Lemma}
\begin{Lemma}[\cite{kks_1998}, Remark 4.8]
	The Ursell kernel $\bar k$ is the unique solution of the so-called non-integrated Kirkwood--Salsburg equation
	\begin{equation}\label{eq:diff:ursell}
	\begin{cases}
		\bar{k}(\gamma,\xi) = e^{-\beta\sum_{\y\in\gamma\setminus\{\x\}}\Phi(\x,\y)}\sum_{\eta\subset\xi}k_{\eta}(\x)\bar{k}\big((\gamma\setminus\{\x\})\eta,\xi\setminus\eta\big)\\
		\bar{k}(\underline{o},\xi) = \1_{\{\xi=\underline{o}\}},
	\end{cases}
	\end{equation}
	where $k_\eta(\x)\defeq \displaystyle\prod_{\y\in\eta}\left(e^{-\beta\Phi(\x,\y)}-1\right)$, and $\x\in\gamma$ is chosen arbitrarily.
\end{Lemma}

We now introduce a second functional $Q$, which satisfies a similar equation to \eqref{eq:diff:ursell}, dominates the Ursell kernel, and whose simpler expression allows for more convenient computations.
\begin{Definition}
	Consider a functional $Q$ on $\confs_f\times\confs_f$ defined as follows: for any $\xi\in\confs_f$, $Q(\underline{o},\xi) = \1_{\{\xi=\underline{o}\}}$, and for any $\gamma=\{\x_1,\dots,\x_N\}$, $N\geq 1$,
\begin{equation*}
	Q(\gamma,\xi) \defeq \sum_{\substack{\xi_1,\dots,\xi_N\subset\xi\\\xi_i\cap\xi_j=\underline{o}\,\forall i\neq j}} Q(\{\x_1\},\xi_1)	\cdots Q(\{\x_N\},\xi_N),
\end{equation*}
where
\begin{equation}\label{eq:diff:Qi}
\begin{cases*}
	Q(\{\x\},\xi)\defeq e^{2\beta\stabconst_\Phi(\abs{\xi}+1)}\sum_{T\in\mathscr{T}(\{x\}\cup\xi)}\prod_{\{\y_1,\y_2\}\in T}\abs{e^{-\beta\Phi(\y_1,\y_2)}-1}\quad \text{if }\xi\neq\underline{o}\\
	Q(\{\x\},\underline{o})=e^{2\beta\stabconst_\Phi}.
\end{cases*}
\end{equation}
\end{Definition}

\begin{Proposition}[\cite{kks_1998}, Proposition 4.10]\label{prop:diff:Q}
The functional $Q$ defined above is the unique solution of
\begin{equation*}
\begin{cases}
	Q(\gamma,\xi) = e^{2\beta\stabconst_{\Phi}}\sum_{\eta\subset\xi}\abs{k_\eta\left(\msb{i}(\gamma)\right)}Q\left(\gamma\setminus\msb{i}(\gamma)\cup\eta,\xi\setminus\eta\right)\\
	Q(\underline{o},\xi) = \1_{\{\xi=\underline{o}\}},
\end{cases}
\end{equation*}
where the functional $\msb{i}$ was defined in \eqref{eq:diff:I}.
\end{Proposition}
\begin{Corollary}[\cite{kks_1998}, Proposition 4.11]\label{cor:diff:QandbarK}
	For any $\gamma=\{\x_1,\dots,\x_N\}$, $N\geq 1$, and $\xi\in\confs_f$ such that $\gamma\cap\xi=\underline{o}$, we have
	\begin{equation*}
	\begin{split}
		\abs{\bar k(\gamma,\xi)}&\leq Q(\gamma,\xi)\\
		&= \sum_{\substack{\xi_1,\dots,\xi_N\subset\xi\\\xi_i\cap\xi_j=\underline{o}\,\forall i\neq j}}Q(\{\x_1\},\xi_1)\dots Q(\{\x_N\},\xi_N),
	\end{split}
	\end{equation*}
	and 
	\begin{equation*}
		\abs{k(\gamma)}\leq e^{2\beta\stabconst_\Phi\abs{\gamma}}\sum_{T\in\mathscr{T}(\gamma)}\prod_{\{\x_i,\x_j\}\in T} \abs{e^{-\Phi(\x_i,\x_j)}-1}.
	\end{equation*}
\end{Corollary}

\begin{Lemma}\label{lem:diff:Qinduction}
	For any finite volume $\Lambda\subset\IR^d$ and $N\geq 1$, for $\lambda$-a.a. $\x\in\states$,
	\begin{equation*}
	\begin{split}
		\int_{(\Lambda\times \marks)^N} &Q(\{\x\},\{\y_1,\dots,\y_N\})\sigmab(d\y_1)\cdots\sigmab(d\y_N)\\
		&\leq e^{2\beta\stabconst_\Phi(N+1)} \RuelleC^{N-1}(N+1)^{N-1} \int_{\Lambda\times \marks}\abs{e^{-\beta\Phi(\x,\y)}-1}\sigmab(d\y).
	\end{split}
	\end{equation*}
\end{Lemma}
\begin{proof}
	Using \eqref{eq:diff:Qi}, we rewrite the l.h.s. as
	\begin{equation*}
		e^{2\beta\stabconst_\Phi(N+1)}\sum_{T\in\mathscr{T}([N+1])}\underbrace{\int_{(\Lambda\times \marks)^N} \prod_{\{i,j\}\in T}\abs{e^{-\beta\Phi(\y_i,\y_j)}-1}\,	\sigmab(d\y_1)\cdots\sigmab(d\y_N)}_{\eqdef\, I_N},
	\end{equation*}
	where we set $\y_{N+1} \defeq \x$, and $[N+1] \defeq \{1,\dots,N+1\}$. We estimate $I_N$ by induction on $N\geq 1$:
	\begin{itemize}
		\item For $N = 1$,
		\begin{equation*}
 			I_1 = \int_{\Lambda\times \marks} \abs{e^{-\beta\Phi(\x,\y_1)}-1}\sigmab(d\y_1).
 		\end{equation*}
		\item For the inductive step, assume that, for all $T\in \mathscr{T}([N])$,
		\begin{equation*}
			\int_{(\Lambda\times \marks)^{N-1}} \prod_{\{i,j\}\in T}\abs{e^{-\beta\Phi(\y_i,\y_j)}-1}\,	\bigotimes_{i=1}^{N-1} \sigmab(d\y_i) \leq \RuelleC^{N-2} \int_{\Lambda\times \marks} \abs{e^{-\beta\Phi(\y_{N},\y)}-1}\sigmab(d\y).
		\end{equation*}
		\item Let $T\in\mathscr{T}([N+1])$ be given, and root it in $\y_{N+1}$. There exists then an edge $\{j_1,j_2\}\in T$, where $\y_{j_1}$ is a leaf, and $\y_{j_1}\neq \y_{j_{N+1}}$. We obtain
		\begin{equation*}
		\begin{split}
			&\int_{(\Lambda\times \marks)^{N}} \prod_{\{i,j\}\in T}\abs{e^{-\beta\Phi(\y_i,\y_j)}-1}\,	 \bigotimes_{i=1}^N \sigmab(d\y_i)\\
			& = \int_{(\Lambda\times \marks)^{N-1}}\underbrace{\int_{\Lambda\times \marks} \abs{e^{-\beta\Phi(\y_{j_1},\y_{j_2})}-1}\, \sigmab(d\y_{j_1})}_{ \leq\, \RuelleC}\prod_{\{i,j\}\in T\setminus\{\{j_1,j_2\}\}} \abs{e^{-\beta\Phi(\y_i,\y_j)}-1}\bigotimes_{\substack{i=1\\i\neq j_1}}^N \sigmab(d\y_i)\\
			&\leq \RuelleC \int_{(\Lambda\times \marks)^{N-1}} \prod_{\{i,j\}\in T\setminus\{\{j_1,j_2\}\}} \abs{e^{-\beta\Phi(\y_i,\y_j)}-1}\bigotimes_{\substack{i=1\\i\neq j_1}}^N\sigmab(d\y_i).
		\end{split}
		\end{equation*}
		We can then use the inductive step to prove the assertion.
	\end{itemize}
	Moreover,
	\begin{equation*}
		e^{2\beta\stabconst_\Phi(N+1)}\sum_{T\in\mathscr{T}([N+1])} I_N \leq e^{2\beta\stabconst_\Phi(N+1)}\sum_{T\in\mathscr{T}([N+1])} \RuelleC^{N-1} \int_{\Lambda\times \marks} \abs{e^{-\beta\Phi(\y_{N+1},\y)}-1}\sigmab(d\y),
	\end{equation*}
	and the claim follows, since the number of elements of $\mathscr{T}([N+1])$ is $(N+1)^{N-1}$ (see Theorem 4.1.3 of \cite{ore_1995}).
\end{proof}

\begin{Lemma}\label{lem:diff:RB1}
	Define the threshold activity
	\begin{equation}\label{eq:diff:zRuelle}
		\zRuelle \defeq (\RuelleC e^{2\beta\stabconst_\Phi+1})^{-1},	
	\end{equation}
	and let $z<\zRuelle$. For any finite volume $\Lambda\subset\IR^d$, for $\tilde\pi^{z\sigmab}$-a.a. $\gamma\in\confs_f$, if $\abs{\gamma}=N\geq 1$,
	\begin{equation*}
		\int_{\confs_\Lambda} \abs{\bar k(\gamma,\xi)}\tilde\pi_\Lambda^{z\sigmab}(d\xi) \leq \constanz^{N},
	\end{equation*}
	where
	\begin{equation}\label{eq:diff:cz}
		\constanz \defeq e^{2\beta\stabconst_\Phi}\left(1 + \frac{e}{\sqrt{2\pi}} \log\left(\frac{1}{1 - z/\zRuelle}\right)\right)<+\infty.
	\end{equation}
	Moreover, for any $z<\zRuelle$,
	\begin{equation}\label{eq:diff:finitek}
		\int_{\confs_\Lambda} \abs{k(\xi)}\tilde\pi_\Lambda^{z\sigmab}(d\xi)<+\infty.
	\end{equation}
\end{Lemma}
\begin{proof}
	Let $\gamma=\{\x_1,\dots,\x_N\}$. From Corollary \ref{cor:diff:QandbarK},
	\begin{equation*}
		\int_{\confs_\Lambda} \abs{\bar k(\gamma,\xi)}\tilde\pi_\Lambda^{z\sigmab}(d\xi)\leq \prod_{i=1}^N \int_{\confs_\Lambda} Q(\{\x_i\},\xi)\tilde\pi_\Lambda^{z\sigmab}(d\xi).
	\end{equation*}
	Thanks to Lemma \ref{lem:diff:Qinduction},
	\begin{equation*}
	\begin{split}
		\int_{\confs_\Lambda}& Q(\{\x_i\},\xi)\tilde\pi_\Lambda^{z\sigmab}(d\xi)= \sum_{N=0}^{+\infty} \frac{z^N}{N!} \int_{(\Lambda\times \marks)^N} Q(\{\x_i\},\{\y_1,\dots,\y_N\})\,\sigmab(d\y_1)\cdots\sigmab(d\y_N) \\
		&= e^{2\beta\stabconst_\Phi} + \sum_{N=1}^{+\infty}\frac{z^N}{N!}e^{2\beta\stabconst_\Phi(N+1)} \RuelleC^{N-1}(N+1)^{N-1} \underbrace{\int_{\Lambda\times \marks}\abs{e^{-\beta\Phi(\x,\y)}-1}\sigmab(d\y)}_{\leq \RuelleC}\\
		&\leq e^{2\beta\stabconst_\Phi}\left(1 + \frac{e}{\sqrt{2\pi}}\sum_{N=1}^{+\infty} \frac{(z\, \RuelleC e^{2\beta\stabconst_\Phi+1})^N}{N^{3/2}}\right)\leq e^{2\beta\stabconst_\Phi}\left(1 + \frac{e}{\sqrt{2\pi}}\sum_{N=1}^{+\infty} \frac{(z\, \RuelleC e^{2\beta\stabconst_\Phi+1})^N}{N}\right),
	\end{split}
	\end{equation*}
	where, in the third step, we used the inequality $(N+1)^{N-1}\leq \frac{1}{\sqrt{2\pi}} e^{N+1}\frac{N!}{(N+1)^{3/2}}$, which is a consequence of Stirling's formula: for any $n\geq 0$, 
	\begin{equation*}
		\sqrt{2\pi}n^{n+1/2}e^{-n}e^{1/(12n+1)}\leq n! \Rightarrow n^{n-2}\leq \frac{1}{\sqrt{2\pi}}e^n \frac{(n-1)!}{n^{3/2}}.
	\end{equation*}
	For $z<(\RuelleC e^{2\beta\stabconst_\Phi+1})^{-1}\eqdef\zRuelle$, the above series converges, and we obtain
	\begin{equation*}
		\int_{\confs_\Lambda} Q(\{\x_i\},\xi)\tilde\pi_\Lambda^{z\sigmab}(d\xi)\leq  e^{2\beta\stabconst_\Phi}\left(1 + \frac{e}{\sqrt{2\pi}} \log\left(\frac{1}{1 - z/\zRuelle}\right)\right)\eqdef \constanz.	
	\end{equation*}
	By using Corollary \ref{cor:diff:QandbarK}, and proceeding similarly to the proof of Lemma \ref{lem:diff:Qinduction}, we obtain that, for $z<\zRuelle$,
	\begin{equation*}
		\int_{\confs_\Lambda}\abs{k(\xi)}\tilde\pi_\Lambda^{z\sigmab}(d\xi) <+\infty.
	\end{equation*}
\end{proof}

\begin{Remark}
	Note that $\constanz$ depends on $z$ but is uniform in $\Lambda$; moreover, $\constan_0=e^{2\beta\stabconst_\Phi}$.	
	We also note how the estimates of the above proof, while not optimal, are essential to the method. This leads to an explicit but potentially not optimal Ruelle bound.
\end{Remark}

\subsection{A Ruelle bound for correlation functions}\label{sec:diff:RB}

As a consequence of \eqref{eq:diff:finitek}, we can use the representation \eqref{eq:diff:corrBark} of the correlation function
\begin{equation*}
	\corr{\Lambda}{z,\beta}(\gamma) = \int_{\confs_\Lambda} \bar k(\gamma,\xi)\tilde\pi_\Lambda^{z\sigmab}(d\xi),
\end{equation*}	
and use the above tree-graph estimates to obtain the following Ruelle bound:
\begin{Proposition}\label{prop:diff:RB}
	Let $\beta>0$ and $\zRuelle$ as defined in \eqref{eq:diff:zRuelle}. For a pair potential $\Phi$ satisfying Assumptions \ref{hyp:diff:0c}+\ref{hyp:diff:2}, for any activity $z\in(0,\zRuelle)$ and any finite volume $\Lambda\subset\IR^d$, the finite-volume correlation function $\corr{\Lambda}{z,\beta}$ satisfies, for $\tilde\pi^{z\sigmab}$-a.a. $\gamma\in\confs_\Lambda$,
	\begin{equation}\label{eq:diff:RB0}
		\corr{\Lambda}{z,\beta}(\gamma) \leq \constanz^{\abs{\gamma}},
	\end{equation}
	where the constant $\constanz$ is defined in \eqref{eq:diff:cz}.	Moreover, a similar bound holds for the $N$-point correlation functions of any $P\in\Gibbs$: for any $z\in(0,\zRuelle)$, for any $N\geq 1$, for $\sigmab^{\otimes N}$-almost all $\{\x_1,\dots,\x_N\}\subset\states^N$
	\begin{equation}\label{eq:diff:RB}
		\corr{N}{P}(\x_1,\dots,\x_N)\leq \constanz^N.
	\end{equation}
\end{Proposition}
\begin{proof}
	Fix $z<\zRuelle$. The first statement is an immediate consequence of Lemma 	\ref{lem:diff:RB1}.
	Moreover, as the right hand side of \eqref{eq:diff:RB0} does not depend on $\Lambda$, this bound also holds in the limit as $\Lambda\uparrow \IR^d$, so for the limiting correlation function $\corr{f}{z}(\gamma) \defeq \int_{\confs_f} \bar k(\gamma,\xi)\tilde\pi^{z\sigmab}(d\xi)$, $\gamma\in\confs_f$.
	
	For the second statement, consider $\gamma = \{\x_1,\dots,\x_N\}$. It is known (see \cite{poghosyan_zessin_2020}, Lemmas 11 and 14), that the limiting correlation functional $\corr{f}{z}(\gamma)$ coincides with the correlation function $\corr{N}{P}(\gamma)$ whenever the expression in \eqref{eq:diff:corrGibbs} is well defined. As this is indeed the case in our setting,
	the Ruelle bound \eqref{eq:diff:RB} holds for any $P\in\Gibbs$.
\end{proof}

\section{Uniqueness via the Kirkwood--Salsburg equations}\label{sec:diff:KS}

We are in the following situation: we have an infinite-volume Gibbs point process $P\in\Gibbs$ associated to a potential $\Phi$ (not necessarily constructed as in Section \ref{sec:diff:existence}) and whose correlation functions satisfy a Ruelle bound (not necessarily that of Section \ref{sec:diff:RBcorrelation}), and wish to understand whether it is indeed the unique such process associated to $\Phi$ and with activity $z$.


The uniqueness proof is structured as follows: we prove that the correlation functions of a Gibbs point process satisfy the Kirkwood--Salsburg equations. Moreover, thanks to the Ruelle bounds, these correlation functions belong to an appropriate Banach space, where these equations have at most one solution. From this, we obtain the uniqueness of the Gibbs point process $P$.

In this section we work under Assumptions \ref{hyp:diff:0b}+\ref{hyp:diff:4}+\ref{hyp:diff:RBnonunif}, introduced in Subsection \ref{sec:diff:generalRB} below.

\subsection{The Kirkwood--Salsburg equations}

The key goal of this section is to show that the correlation functions $(\corr{N}{P})_N$ of any $P\in\Gibbs$ solve, for all $N\geq 1$, for $\sigmab^{\otimes (N+1)}$-almost all $(\x_0,\dots,\x_N)\in\states^{N+1}$, the sequence of \emph{Kirkwood--Salsburg equations}
\begin{equation}\label{eq:diff:KS}\tag*{(KS)$_z$}
\begin{split}
	&\corr{N+1}{P}(\x_0,\dots,\x_N) = e^{-\beta\CondEn{\x_0}{\x_1,\dots,\x_N}} \big( \corr{N}{P}(\x_1,\dots\x_N) \\
	&\quad+ \sum_{k=1}^{+\infty}\frac{z^k}{k!}\int \prod_{j=1}^k(e^{-\beta\Phi(\x_0,\y_j)}-1)\corr{N+k}{P}(\x_1,\dots,\x_N,\y_1,\dots,\y_k)\sigmab^{\otimes k}(d\y_1,\dots,d\y_k)\big),
\end{split}
\end{equation}	
where, by convention, $\corr{0}{P}=1$.

We can interpret the Kirkwood--Salsburg equations as an operator acting on the Banach space $\Banach$, defined below.
\begin{Definition}
	The Banach space $\Banach$ is the set of all sequences $r=(r_N)_N$ such that
	\begin{equation*}
		\exists\, \tg{b}_r\geq 0: \forall N\geq 1,\ \abs{r_N(\x_1,\dots,\x_N)}\leq \tg{b}_r \prod_{i=1}^N\constan(\x_i),
	\end{equation*}
	endowed with the norm $\norm{r}_\constan$ equal to the smallest such $\tg{b}_r$.
\end{Definition}
Note that, in the case of $\constan>0$ constant, the right hand side reads $\tg{b}_r \constan^N$.

\begin{Definition}
	Consider the \emph{Kirkwood--Salsburg operator} $\K_z$, $z>0$, acting on $\Banach$, given by
	\begin{align}\label{eq:diff:KSoperator}
		&(\K_z r)_1(\x_0) = \sum_{k=1}^{+\infty}\frac{z^k}{k!}\int \prod_{j=1}^k(e^{-\beta\Phi(\x_0,\y_j)}-1)r_{N+k}(\x_1,\dots,\x_N,\y_1,\dots,\y_k)\ \sigmab^{\otimes k}(d\y_1,\dots,d\y_k)\big);\nonumber \\
		&(\K_z r)_{N+1}(\x_0,\dots,\x_N) = e^{-\beta\sum_{i=1}^N\Phi(\x_0,\x_i)} \big( r_N(\x_1,\dots\x_N) \\
		&+ \sum_{k=1}^{+\infty}\frac{z^k}{k!}\int \prod_{j=1}^k(e^{-\beta\Phi(\x_0,\y_j)}-1)r_{N+k}(\x_1,\dots,\x_N,\y_1,\dots,\y_k)\ \sigmab^{\otimes k}(d\y_1,\dots,d\y_k)\big),\ N\geq 1 \nonumber.
	\end{align}
\end{Definition}
The Kirkwood--Salsburg equations \ref{eq:diff:KS} can now be rewritten as the following fixed-point problem in the Banach space $\Banach$:
	\begin{equation}\label{eq:diff:fixed}
		r = \K_z r + \underline{1}_z,	
	\end{equation}
	where $\underline{1}_z = (\underline{1}_{z,N})_N$ is given by $\underline{1}_{z,1}(\x_1) = 1$, $\underline{1}_{z,N} = 0$ for $N\geq 2$. 


\subsection{The Kirkwood--Salsburg operator}\label{sec:diff:RBnonunif}

In this subsection we show that, under the assumptions presented below -- in particular when we already know that the correlation functions of a Gibbs point process satisfy a Ruelle bound -- there exists an activity threshold $\zCrit>0$ such that, for any $z<\zCrit$, the Kirkwood--Salsburg operator $\K_z$ is a contraction.

\subsubsection{The general case}\label{sec:diff:generalRB}

We allow for a weakening of Assumptions \ref{hyp:diff:0} and \ref{hyp:diff:2}. See Subsection \ref{sec:diff:RBunif} for a comparison, when assuming that a \emph{uniform} Ruelle bound holds.

\setcounter{Assumption}{1}
\begin{AssumptionOpt}{$'$}[Weak stability]\label{hyp:diff:0b}
	Consider an energy functional $\Energy$ as in \eqref{eq:diff:energy}, where the self-potential $\Psi$ satisfies \eqref{eq:diff:Psi}, but for which the stability condition \hyp{t.st.} of the pair potential $\Phi$ is replaced by a weaker one:
\begin{description}[font = \normalfont,align=left]
	\item[\hyp{w.st.}] The pair potential between two \comment{marked points} is given by a symmetric functional $\Phi:\states\times\states\rightarrow\IRi$ such that there exists a function $\msb{b}\colon\states\to\IR_+$ where, for any $\{\x_0,\dots,\x_N\}\subset\states$ and some $\x\in \{\x_0,\dots,\x_N\}$ (w.l.o.g. $\x_0$), the following holds
	\begin{equation}\label{eq:diff:Pstab}
		\sum_{i=1}^N\Phi(\x_0,\x_i)\geq -\msb{b}(\x_0).
	\end{equation}
\end{description}	
\end{AssumptionOpt}
 
\begin{Remark}\label{rmk:diff:stab}
	\comment{Note that \hyp{w.st.} is actually implied by the conditional stability condition \eqref{eq:diff:condStab}.}
	Moreover, if $\Phi$ satisfies \hyp{t.st.} [resp. \hyp{st.}], then \hyp{w.st.} holds for $\msb{b}(x,m)=2\stabconst_\Phi(1+\norm{m}^{d+\delta})$ [resp. $2\stabconst_\Phi$, see \eqref{eq:diff:I}]. In other words, we have the following implications: \hyp{st.} $\Rightarrow$ \hyp{t.st.} $\Rightarrow$ \hyp{w.st.} $\Leftarrow$ \hyp{t.cond.st.}.

	Conversely, if $\Phi$ satisfies \hyp{w.st.} for $\msb{b}(x,m) = \stabconst_\Phi(1+\norm{m}^{d+\delta})$ [resp. $\msb{b}\equiv \stabconst_\Phi$], then is also satisfies \hyp{t.st.} [resp. \hyp{st.}] for the same function [resp. constant].
\end{Remark}

While in the previous section we assumed that the pair potential $\Phi$ satisfied a \emph{uniform} regularity condition (Assumption \ref{hyp:diff:2}), here we work with potentials $\Phi$ that satisfy the following \emph{weighted} regularity condition (cf. \cite{poghosyan_ueltschi_2009}):
\setcounter{Assumption}{3}
\begin{AssumptionOpt}{$'$}[Weighted regularity]\label{hyp:diff:4}
	There exist a positive function $\msb{a}\colon\states\to [a,+\infty)$ and 
	a critical activity $\zCrit>0$ such that, for any $\x\in\states$,
	\begin{equation}\label{eq:diff:reg}
		\zCrit\int e^{\msb{a}(\y)+ \msb{b}(\y)}\abs{ e^{-\beta\Phi(\x,\y)}-1} \ \sigmab(d\y) \leq \msb{a}(\x),
	\end{equation}
	with $\msb{a}+\msb{b}$ locally integrable in the following sense: for any bounded $A\subset \states$,
	\begin{equation}\label{eq:diff:abinteg}
		\int_{A} e^{\msb{a}(\x)+\msb{b}(\x)}	\sigmab(d\x) < +\infty.
	\end{equation}
\end{AssumptionOpt}
\begin{Remark}
	Under Assumptions \ref{hyp:diff:0c}+\ref{hyp:diff:2}, one can choose $\msb{a}\equiv 1$ and $\msb{b}\equiv 2\beta \stabconst_\Phi$, yielding $\zCrit = (e\RuelleC e^{2\beta \stabconst_\Phi})^{-1}$. Note that a Ruelle bound does not automatically follow.
	
	The local integrability assumption \eqref{eq:diff:abinteg} is needed for the method of moments in step \emph{(iv)} of the proof of Theorem \ref{thm:diff:uniqueness}.
\end{Remark}

Let $\msb{a}$ and $\msb{b}$ as above. We also assume that the correlation functions satisfy a Ruelle bound of the following form:

\setcounter{Assumption}{4}
\begin{AssumptionOpt}{$'$}[Non-uniform Ruelle bound]\label{hyp:diff:RBnonunif}
	For any $P\in\Gibbs$, for any $N\geq 1$, for $\sigmab^{\otimes N}$-almost all $\{\x_1,\dots,\x_N\}\subset\states^N$, the following holds:
	\begin{equation}\label{eq:diff:RBnonunif}
		\corr{N}{P}(\x_1,\dots,\x_N)\leq \prod_{i=1}^N e^{\msb{a}(\x_i) + \msb{b}(\x_i)}.
	\end{equation}
\end{AssumptionOpt}

\begin{Remarks}
The Ruelle bound assumption guarantees that the sequence $\rho^{(P)}=(\corr{N}{P})_N$ of correlation functions of any $P\in\Gibbs$ is an element of the Banach space $\Banach$, for $\constan(\x)=e^{\msb{a}(\x)+\msb{b}(\x)}$.

The choice of functions $\msb{a}$ and $\msb{b}$ has to take into account all three of the above assumptions, so some tuning is needed; in particular, the Ruelle bound has to hold for precisely $\constan(\x)\defeq e^{\msb{a}(\x)+\msb{b}(\x)}$. Indeed, in order to prove that a Ruelle bound for such a choice of $\msb{c}$ holds for the correlation functions of any Gibbs point process, the proof becomes more technical (see Corollary 2 and Theorems 2 and 3 of \cite{poghosyan_zessin_2020}). It is however true that, under some additional uniform assumption (either for the conditional energy, as in Example \ref{ex:diff:reg} below, or for the Ruelle bound, as in Subsection \ref{sec:diff:RBunif}) one is able to tune the constants in such a way that $\msb{c} = e^{\msb{a}+\msb{b}}$.


\end{Remarks}

\setcounter{Example}{2}
\begin{Example}\label{ex:diff:reg}
	Consider a potential $\phi = \phi_{hc} + \phi_l$, satisfying Assumptions \ref{hyp:diff:0c}+\ref{hyp:diff:1}, given by the sum of a hard-core potential (with hard-core diameter $R>0$) and a bounded potential $\phi_l$, on $[R,+\infty)$:
	\begin{equation*}
		\exists M_\phi>0: \phi_l(u)\leq M_\phi \ \forall u\geq R.
	\end{equation*}
	In particular, we recall from \eqref{eq:diff:condStab} that there exists a constant $\bar\stabconst_\Phi\geq 0$ such that, for any $\x=(x,m)\in\states$, for any $\xi\in\confs^{\tempered}$,
 	\begin{equation}\label{eq:diff:excondStab}
 		\CondEn{\x}{\xi} \geq -\bar \stabconst_\Phi(1+\norm{m}^{d+\delta}).
	\end{equation}
	We show here that there exist functions $\msb{a}$ and $\msb{b}$, and a threshold activity $\zCrit>0$ such that Assumptions \ref{hyp:diff:4} and \ref{hyp:diff:RBnonunif} hold for any $z\in(0,\zCrit)$.
	\begin{proof}
	Using \eqref{eq:diff:basicbound}, we have $\lvert e^{-\beta\Phi}-1\rvert \leq \abs{\beta\bar \Phi}e^{2\beta\stabconst_\Phi}$, and the weighted regularity condition follows as soon as
	\begin{equation*}
	\begin{split}
		ze^{2\beta\stabconst_\Phi} \int_\states e^{\msb{a}(\x_2)}\int_0^{\interval} \beta\abs{\bar\phi(x_2 + m_2(s) - x_1 - m_1(s))}ds\ \1_{\{\lvert x_2-x_1\rvert\leq a_0 + \norm{m_2} + \norm{m_1}\}} \, \sigmab(d\x_2) \leq \msb{a}(\x_1).
	\end{split}
	\end{equation*}
	Considering a function $\msb{a}$ of the form $\msb{a}(x,m)=\msb{a}(m) = A(1+ \norm{m}^{d+\delta})$, for some constant $A>0$ to be determined, and recalling that $\Psi(x,m)\geq -\constPsi\norm{m}^{d+\delta}$, this reduces to
	\begin{equation*}
	\begin{split}
		z\beta e^{2\beta\stabconst_\Phi} \int_{\marks} e^{A(1 + \norm{m_2}^{d+\delta})} \int_{\IR^d}\int_0^1 \big\lvert \phi\big(\abs{x_1+m_1(s) - x_2 - m_2(s)}\big)\big\rvert ds\1_{\{\lvert x_1-x_2\rvert\leq a_0 + \norm{m_1} + \norm{m_2}\}}\\
		dx_2\ e^{\constPsi\norm{m_2}^{d+\delta}} \refmark(dm_2)
		 \leq A(1+ \norm{m_1}^{d+\delta}).
	\end{split}
	\end{equation*}
	Estimating the left hand side leads to:
	\begin{equation*}
		z\beta e^{2\beta\stabconst_\Phi} \int_{\marks} \left(b_d R^d + M_\phi k_d b_d (a_0^d + \norm{m_1}^d + \norm{m_2}^d)\right) e^{A(1 + \norm{m_2}^d) + \constPsi\norm{m_2}^{d+\delta}} \refmark(dm_2),
	\end{equation*}
	where $k_d$ is such that $(x+y+z)^d\leq k_d(x^d+y^d+z^d)$, and $b_d$ the volume of the unit ball in $\IR^d$. Setting 
	\begin{equation*}
		\upsilon_A:=\int e^{A(1+\norm{m}^{d+\delta}) + \constPsi \norm{m}^{d+2\delta}} \refmark(dm),	
	\end{equation*}
	which is finite thanks to the definition of the measure $R$, we can fix $A$ by the following (note that $A\geq \bar\stabconst_\Phi$; the reason for this choice will be apparent shortly):
	\begin{equation*}
		A \defeq \sup_{u\geq 0} \frac{\bar\stabconst_\Phi(1+u^{d+\delta})\ \vee\ b_d\left(R^d + M_\phi k_d (a_0^d + u^d + 1)\right)}{1+u^{d+\delta}} < +\infty,	
	\end{equation*}
	so that the regularity assumption is satisfied for $\msb{a}(x,m) = A(1 + \norm{m_1}^{d+\delta})$, $\msb{b} \equiv 2\stabconst_\Phi$, and
	\begin{equation*}
		\zCrit\defeq (\upsilon_A \beta e^{2\beta\stabconst_\Phi})^{-1}.
	\end{equation*}
	
	From the representation \eqref{eq:diff:corrGibbs} of the correlation functions and \eqref{eq:diff:excondStab}, the Ruelle bound of Assumption \ref{hyp:diff:RBnonunif} follows as well, since $A\geq \bar\stabconst_\Phi$ by construction.
 	\end{proof}
\end{Example}

\begin{Proposition}\label{prop:diff:corrKS}
	Let $\beta>0$, $z>0$. Under Assumptions \ref{hyp:diff:0b}+\ref{hyp:diff:4}+\ref{hyp:diff:RBnonunif}, the correlation functions $(\corr{N}{P})_N$ of any $P\in\Gibbs$ solve, for all $N\geq 1$, for $\sigmab^{\otimes (N+1)}$-almost all $(\x_0,\dots,\x_N)\in\states^{N+1}$, the \emph{Kirkwood--Salsburg equation} \ref{eq:diff:KS} defined above.
\end{Proposition}
\begin{proof}
	
	We note first that the absolute convergence of the right hand side of \eqref{eq:diff:KSoperator} is guaranteed by the Ruelle bound and the regularity condition. Indeed, for $z<\zCrit$, one has
	\begin{equation*} 
	\begin{split}
		&\sum_{k=1}^{+\infty}\frac{z^k}{k!}\int \prod_{j=1}^k\abs{e^{-\beta\Phi(\x_0,\y_j)}-1}\corr{N+k}{P}(\x_1,\dots,\x_N,\y_1,\dots,\y_k)\ \sigmab^{\otimes k}(d\y_1,\dots,d\y_k)\\
		&\overset{\mathclap{\eqref{eq:diff:RBnonunif}}}{\leq}\ \ \prod_{i=1}^{N}e^{\msb{a}(\x_i)+\msb{b}(\x_i)}\sum_{k=1}^{+\infty}\frac{z^k}{k!}\int \prod_{j=1}^k e^{\msb{a}(\y_j)+ \msb{b}(\y_j)}\abs{ e^{-\beta\Phi(\x_i,\y_j)}-1} \ \sigmab^{\otimes k}(d\y_1,\dots,d\y_k)\\
		&\overset{\mathclap{\eqref{eq:diff:reg}}}{\leq}\ \ \prod_{i=1}^{N}e^{\msb{a}(\x_i)+\msb{b}(\x_i)} \sum_{k=1}^{+\infty}\frac{(\zCrit\msb{a}(\x_i))^k}{k!} = \prod_{i=1}^{N}e^{\msb{a}(\x_i)+\msb{b}(\x_i)+\zCrit\msb{a}(\x_i)}.
	\end{split}
	\end{equation*}
	Consider the $(N+1)$-point correlation function of a Gibbs point process $P$:
	\begin{equation*}
	\begin{split}
		&\corr{N+1}{P}(\x_0,\dots,\x_N) = e^{-\beta\PairEnergy(\x_0,\dots,\x_N)} \int_\confs e^{-\beta\CondEn{\x_0,\dots,\x_N}{\xi}}P(d\xi)\\
		&= e^{-\beta\CondEn{\x_0}{\x_1,\dots,\x_N}} e^{-\beta\PairEnergy(\x_1,\dots,\x_N)} \int_\confs e^{-\beta\CondEn{\x_0}{\xi}} e^{-\beta\CondEn{\x_1,\dots,\x_N}{\xi}}P(d\xi).
	\end{split}
	\end{equation*}
	Using the factorial measure $\xi^{(k)}$, we have the following expansion:
	\begin{equation*}
		e^{-\beta\CondEn{\x_0}{\xi}} = 1 + \sum_{k=1}^{+\infty}\frac{1}{k!} \int_{\states^k} \prod_{j=1}^k(e^{-\beta\Phi(\x_0,\y_j)}-1)\, \xi^{(k)}(d\y_1,\dots,d\y_k),
	\end{equation*}
	which is indeed absolutely convergent, since using the Georgii--Nguyen--Zessin equations (see \cite{nguyen_zessin_1979}) one has:
	\begin{equation}\label{eq:diff:unifConvergent}
	\begin{split}
		\int_\confs& \left(1 + \sum_{k=1}^{+\infty}\frac{1}{k!} \int_{\states^k} \prod_{j=1}^k\abs{e^{-\beta\Phi(\x_0,\y_j)}-1}\, \xi^{(k)}(d\y_1,\dots,d\y_k)\right)\, P(d\xi)\\
		&\overset{\mathclap{\text{(GNZ)}}}{=}\ \ 1 + \sum_{k=1}^{+\infty}\frac{z^k}{k!}\int_{\states^k} \prod_{j=1}^k \abs{e^{-\beta\Phi(\x_0,\y_j)}-1}\, e^{-\beta \PairEnergy(\y_1,\dots,\y_k)}\\
		&\phantom{1 + \sum_{k=1}^{+\infty}\frac{z^k}{k!}\int_{\states^k}} \int_\confs e^{-\beta\CondEn{\y_1,\dots,\y_k}{\xi}} P(d\xi)\, \sigmab^{\otimes k}(d\y_1,\dots,d\y_k)\\
		&=\ \  1 + \sum_{k=1}^{+\infty}\frac{z^k}{k!} \int_{\states^k} \prod_{j=1}^k\abs{e^{-\beta\Phi(\x_0,\y_j)}-1}\, \corr{k}{P}(\y_1,\dots,\y_k)\sigmab^{\otimes k}(d\y_1,\dots,d\y_k)<+\infty,
	\end{split}
	\end{equation}
	where we have used the Ruelle bound and the regularity assumption as before, to prove the integral is finite.
	We can then exchange summation over $k$ and integration over $\confs$, yielding
	\begin{equation*}
	\begin{split}
		&e^{-\beta\PairEnergy(\x_1,\dots,\x_N)} \int_\confs e^{-\beta\CondEn{\x_0}{\xi}} e^{-\beta\CondEn{\x_1,\dots,\x_N}{\xi}}P(d\xi)\\
		&= e^{-\beta\PairEnergy(\x_1,\dots,\x_N)} \int_\confs e^{-\beta\CondEn{\x_1,\dots,\x_N}{\xi}}\, P(d\xi)\\
		&\phantom{1 + }+ \sum_{k=1}^{+\infty}\frac{1}{k!}\int_\confs \int_{\states^k}e^{-\beta\PairEnergy(\x_1,\dots,\x_N) - \beta\CondEn{\x_1,\dots,\x_N}{\xi}} \prod_{j=1}^k(e^{-\beta\Phi(\x_0,\y_j)}-1)\, \xi^{(k)}(d\y_1,\dots,d\y_k)\, P(d\xi)\\
		&= \corr{N}{P}(\x_1,\dots,\x_N)\\
		&\phantom{1 + }+ \sum_{k=1}^{+\infty}\frac{z^k}{k!}\int_{\states^k}
		e^{-\beta \PairEnergy(\x_1,\dots,\x_N,\y_1,\dots,\y_k)}
		\prod_{j=1}^k(e^{-\beta\Phi(\x_0,\y_j)}-1)\\
		&\phantom{1 + \sum_{k=1}^{+\infty}\frac{1}{k!} \int_{\states^k}}\int_\confs e^{-\beta\CondEn{\x_1,\dots,\x_N,\y_1,\dots,\y_k}{\xi}} \, P(d\xi)\, \sigmab^{\otimes k}(d\y_1,\dots,d\y_k)\\
		&= \corr{N}{P}(\x_1,\dots,\x_N)\\
		&\phantom{1 + }+ \sum_{k=1}^{+\infty}\frac{z^k}{k!} \int_{\states^k}\prod_{j=1}^k(e^{-\beta\Phi(\x_0,\y_j)}-1)\corr{N+k}{P}(\x_1,\dots,\x_N,\y_1,\dots,\y_k)\, \sigmab^{\otimes k}(d\y_1,\dots,d\y_k),
	\end{split}
	\end{equation*}
	and concluding the proof.
\end{proof}
	
\begin{Proposition}\label{prop:diff:KSX}
	Let $\Phi$ such that Assumptions \ref{hyp:diff:0b}+\ref{hyp:diff:4}+\ref{hyp:diff:RBnonunif} hold, and set $\constan(\x)\defeq e^{\msb{a}(\x) + \msb{b}(\x)}$. For any $\beta>0$ and $z\in(0,\zCrit)$, the operator $\K_z$ is a contraction in $\Banach$. For such activities there exists then at most one solution of \ref{eq:diff:KS} in $\Banach$.
\end{Proposition}
\begin{proof}
	For any $r\in \Banach$, with $\norm{r}_\constan\leq 1$, we estimate
	\begin{equation*}
	\begin{split}
		\abs{(\K_z r)_{N+1}&(\x_0,\dots,\x_N)}\leq e^{-\sum_{i=1}^N\Phi(\x_0-\x_i)} \bigg(\prod_{i=1}^N\constan(\x_i)\\
		&+ \sum_{k=1}^{+\infty}\frac{z^k}{k!}\int \prod_{j=1}^k\abs{e^{-\Phi(\x_0,\y_j)}-1} \prod_{i=1}^N\constan(\x_i) \prod_{j=1}^k\constan(\y_j)\ \sigmab^{\otimes k}(d\y_1,\dots,d\y_k)\bigg)\\
		&\overset{\eqref{eq:diff:Pstab}}{\leq} e^{\msb{b}(\x_0)}  \Big(1 + \sum_{k=1}^{+\infty}\frac{z^k}{k!}\int \prod_{j=1}^k \constan(\y_j)\abs{ e^{-\Phi(\x_0,\y_j)}-1} \ \sigmab^{\otimes k}(d\y_1,\dots,d\y_k)\Big) \prod_{i=1}^N\constan(\x_i)\\
		&= e^{\msb{b}(\x_0)}  \Big(1 + \sum_{k=1}^{+\infty}\frac{z^k}{k!}\int \prod_{j=1}^k e^{\msb{a}(\y)+\msb{b}(\y)}\abs{ e^{-\Phi(\x_0,\y_j)}-1} \ \sigmab^{\otimes k}(d\y_1,\dots,d\y_k)\Big)\prod_{i=1}^N\constan(\x_i)\\
		&\overset{\eqref{eq:diff:reg}}{\leq} e^{\msb{b}(\x_0)} \sum_{k=0}^{+\infty}\frac{(z/\zCrit)^k\msb{a}^k(\x_0)}{k!} \prod_{i=1}^N\constan(\x_i)\\
		&= e^{\msb{b}(\x_0)} e^{\msb{a}(\x_0)z/\zCrit} \prod_{i=1}^N\constan(\x_i)\leq e^{-a}\prod_{i=0}^N\constan(\x_i).
	\end{split}
	\end{equation*}
	The Kirkwood--Salsburg operator is then a contraction: $|||\K_z|||_\constan < 1$, so that \eqref{eq:diff:fixed} has at most one fixed point, and the associated Kirkwood--Salburg equations \ref{eq:diff:KS} have at most one solution.
\end{proof}
 
\subsubsection{The case of uniform Ruelle bounds}\label{sec:diff:RBunif}

In Subsection \ref{sec:diff:RB} we presented techniques to obtain, under relatively weak assumptions for the interaction, an activity regime in which a uniform Ruelle bound holds for a constant $\constanz>0$. A natural question is then whether the framework presented above applies to this case as well.

We work here with energy functionals $\PairEnergy$ and activities $z>0$ such that Assumptions \ref{hyp:diff:0c}+\ref{hyp:diff:2}+\ref{hyp:diff:RB} hold, i.e. with the classical stability \hyp{st.}.

We consider the case of a Ruelle bound that holds for a constant $\constan$, \emph{uniformly} in the points $\x_1,\dots,\x_N$.
\comment{Note that, although under Assumptions \ref{hyp:diff:0c}+\ref{hyp:diff:2} one can compute some constant $\constanz$ such that the Ruelle bound \eqref{eq:diff:unifRB} is satisfied (Proposition \ref{prop:diff:RB}), this constant may not be optimal. Indeed, in some cases one already has a Ruelle bound -- not obtained via the techniques presented in Subsection \ref{sec:diff:RB} -- which yields a larger uniqueness regime (see Example \ref{ex:diff:RB} at the end of this subsection). For this reason, we introduce the following:}
\addtocounter{Assumption}{0}
\begin{AssumptionOpt}{}[Uniform Ruelle bound]\label{hyp:diff:RB}
	Assume there exist a constant $\constan>0$ such that, for any $P\in\Gibbs$, for any $N\geq 1$, for $\sigmab^{\otimes N}$-almost all $\{\x_1,\dots,\x_N\}\subset\states^N$, its correlation function $\corr{N}{P}$ satisfy, uniformly in $\{\x_1,\dots,\x_N\}$, the following Ruelle bound:
	\begin{equation}\label{eq:diff:unifRB}
		\corr{N}{P}(\x_1,\dots,\x_N)\leq \constan^N.	
	\end{equation}
\end{AssumptionOpt}

We want to show that this set of assumptions yields results which are consistent with those presented in the non-uniform case, in particular Propositions \ref{prop:diff:corrKS} and \ref{prop:diff:KSX}.

Indeed, the computations of Propositions \ref{prop:diff:corrKS} and \ref{prop:diff:KSX} are the same, except for using Assumption \ref{hyp:diff:2} to prove the absolute convergence of the series in \eqref{eq:diff:unifConvergent}. To prove that the operator is a contraction, following the proof of Proposition \ref{prop:diff:KSX} with $\msb{b}(\x)\equiv 2\beta \stabconst_\Phi$, and using the uniform Ruelle bound \eqref{eq:diff:unifRB} in place of \eqref{eq:diff:RBnonunif} yields
\begin{equation*}
\begin{split}
	\abs{(\K_z r)_{N+1}&(\x_0,\dots,\x_N)}\leq \constan^{N+1} \constan^{-1}e^{2\beta\stabconst_\Phi + z\RuelleC\constan}.
\end{split}
\end{equation*}
Hence, if there exists $\zCrit>0$ such that
\begin{equation}\label{eq:diff:unifZ}
	\constan^{-1}e^{2\beta\stabconst_\Phi + \zCrit\RuelleC\constan}<1,	
\end{equation}
the operator is a contraction. Note that this condition is equivalent to Assumption \ref{hyp:diff:4} for constant values of $\msb{a}\equiv a>0,\msb{b}\equiv b>0$. More precisely: since we want a Ruelle bound to hold for $e^{a+b}$, set $a\defeq \log(\constan)-b$, if $\constan>e^b$. Condition \eqref{eq:diff:unifZ} then reads
\begin{equation*}
	\constan^{-1}e^{2\beta\stabconst_\Phi + \zCrit\RuelleC\constan} <1 \iff b+\zCrit\RuelleC\constan < \log(\constan)\iff \zCrit\leq a(\RuelleC\constan)^{-1},
\end{equation*}	
which corresponds to \eqref{eq:diff:reg} for this choice of constants.

\begin{Remark}
	If $\constan\leq e^b$, e.g. for non-negative potentials, one can optimise the value of $\constan$. See Example \ref{ex:diff:RB} below.
	
	Without \eqref{eq:diff:unifZ}, one can only prove that the norm of the Kirkwood--Salsburg operator is bounded, but not that it is smaller than $1$. We present below two examples in which such a $\zCrit>0$ exists.
\end{Remark}


%

%

\setcounter{Example}{3}
\begin{Example}
	The uniform Ruelle bound holds, under Assumption \ref{hyp:diff:2}, with $\constan = \constanz $ as defined in \eqref{eq:diff:cz}, for any $z<\zRuelle$. Moreover, set $f(z)\defeq \frac{e^{2\beta\stabconst_\Phi + z\constanz\RuelleC}}{\constanz}$. We have $f(0)=1$ and
	\begin{equation*}
		f'(z) = e^{2\beta\stabconst_\Phi}\frac{e^{z\constanz\RuelleC}}{\constanz^2}\big(\RuelleC(\constanz^2 + z\constanz' \constanz - \constanz'\big).
	\end{equation*}
	so that $f'(0)<0$. Indeed,
	\begin{equation*}
		\text{sign } f'(0) = \text{sign }\big(\RuelleC\constan_0^2 - \constan_0'\big) = \text{sign } \big(\RuelleC e^{4\beta\stabconst_\Phi}(1 - e^2/\sqrt{2\pi})\big) = -1.
	\end{equation*}
	(see Figures \ref{fig:zplot2} and \ref{fig:zplot1}), The set $\{z>0: \constanz^{-1}e^{2\beta\stabconst_\Phi + z\constanz\RuelleC}<1\}$ is then non-empty, and defining
	\begin{equation*}
		\zCrit\defeq \inf \{z>0: \constanz^{-1}e^{2\beta\stabconst_\Phi + z\constanz\RuelleC}>1\}>0,
	\end{equation*}
	we have that, for any $z<\zCrit$, the norm of $\K_z$ in $\Banachz$ is smaller than $1$, i.e. is a contraction in $\Banachz$.
	
	Finally, note that, since $\lim_{z\to\zRuelle^-}\constanz = +\infty$ and $\constanz^{-1}e^{2\beta\stabconst_\Phi + z\constanz\RuelleC}=+\infty$ for $z\geq \zRuelle$, we have that $\zCrit\leq\zRuelle$. 
	\begin{figure}[ht]
\begin{center}
\begin{minipage}{.9\textwidth}
\begin{minipage}[t]{.45\textwidth}
	\includegraphics[width=\textwidth]{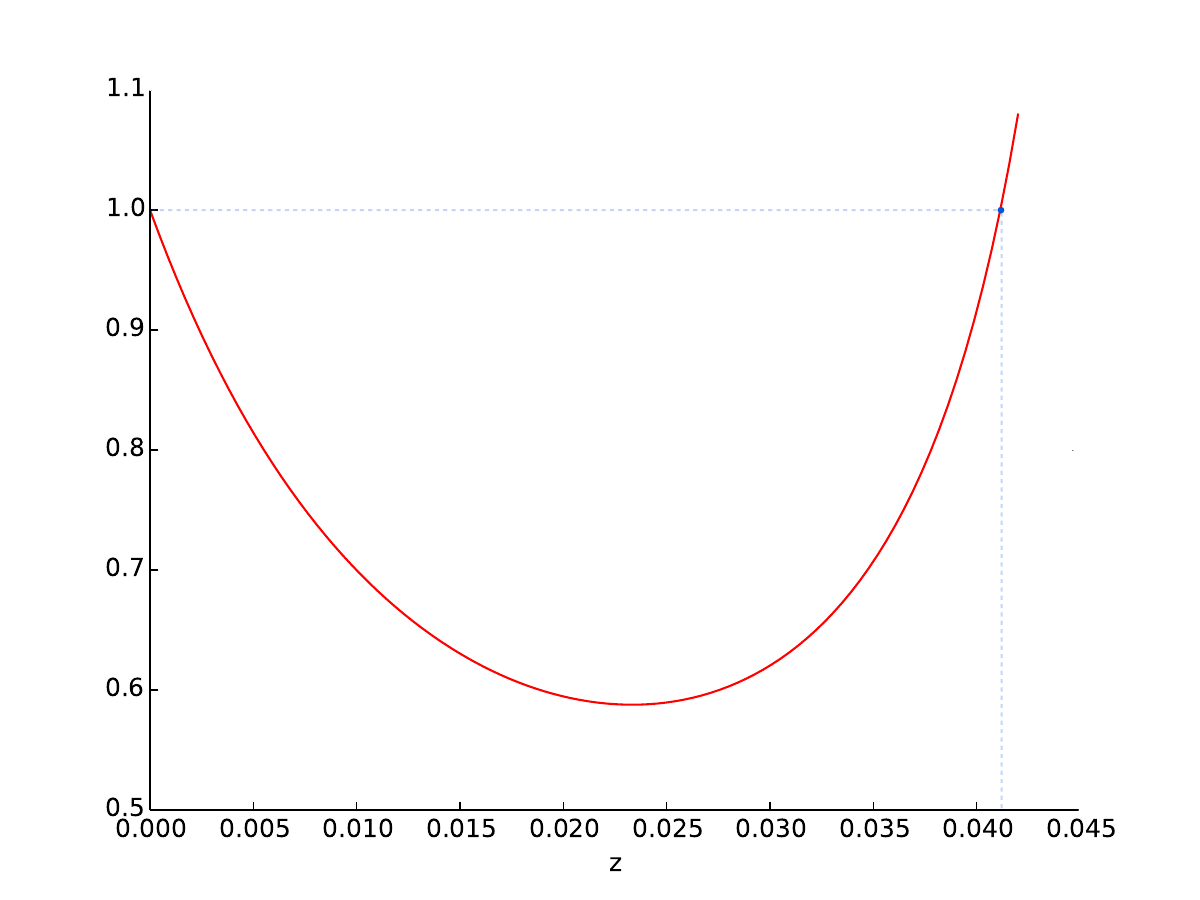}
		\captionof{figure}{Let $\stabconst_\Phi =1$, $\beta=1$, $\mathscr{C}(1) = 1$. Plot of $z\mapsto\constanz^{-1}e^{2 + z\constanz}$. The curve explodes as $z$ approaches $\tg{z}_{Ru}(1)\simeq 0.05$, and the uniqueness domain is $(0,\tg{z}_{\text{crit}}(1))$, where $\tg{z}_{\text{crit}}(1)\simeq 0.041$.}\label{fig:zplot2}
\end{minipage}\hfill
\begin{minipage}[t]{.45\textwidth}
	\includegraphics[width=\textwidth]{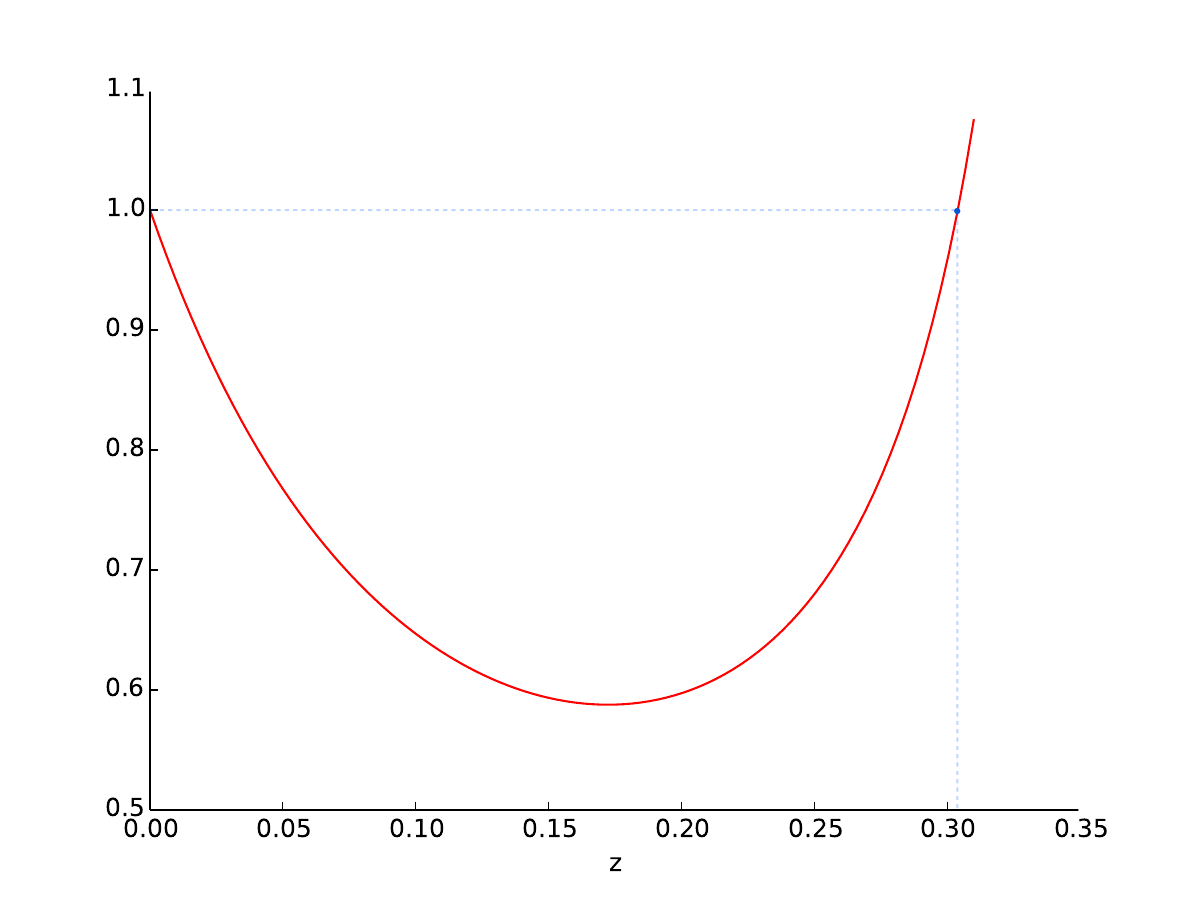}
	\captionof{figure}{Let $\stabconst_\Phi =0$ (i.e. a repulsive potential), $\beta=1$, $\mathscr{C}(1) = 1$. Plot of $z\mapsto\constanz^{-1}e^{z\constanz}$. The curve explodes as $z$ approaches $\tg{z}_{Ru}(1) = 1/e\simeq 0.37$, and the uniqueness domain is $(0,\tg{z}_{\text{crit}}(1))$, where $\tg{z}_{\text{crit}}(1)\simeq 0.304$.}\label{fig:zplot1}
\end{minipage}
\end{minipage}
\end{center}
\end{figure}
\end{Example}

\setcounter{Example}{1}
\begin{Excont}[continued]\label{ex:diff:RB}
	Consider a potential $\Phi$ in the class of Example \ref{ex:diff:potential2}. The Ruelle bound is satisfied, for any $z>0$, for $\constan = e^{3\beta\stabconst_\Phi}$ (see \eqref{eq:diff:exunifRB}).
	
	Let $\stabconst_\Phi>0$. For such a value of $\constan$, the Kirkwood--Salsburg operator $\K_z$ on $\Banach$ is a contraction as soon as $\constan^{-1}e^{2\beta\stabconst_\Phi + z\constan\RuelleC}<1$, that is for $z<\tg{z}_{\text{crit}}(1)=\beta\stabconst_\Phi(\RuelleC e^{3\beta\stabconst_\Phi})^{-1}$. With the choice of constants of Figure \ref{fig:zplot2}, one has $\tg{z}_{\text{crit}}(1)\simeq 0.050$. Notice how in this case, as we predicted at the beginning of this subsection, Assumption \ref{hyp:diff:RB} with $\constan=e^{3\beta\stabconst_{\Phi}}$ yields a larger uniqueness regime than by just using the Ruelle bound with $\constanz$.
	
	\comment{Note that, for repulsive potentials, a Ruelle bound $\msb{c}=1$ trivially holds, so one can actually find that the optimal domain is actually obtained for $\msb{c}=e}$, yielding, for $\beta=1$, $\mathscr{C}(1) = 1$, the critical activity $\tg{z}_{\text{crit}}(1)=1/e=\tg{z}_{Ru}(1)$ (cf. Figure \ref{fig:zplot1}). This is indeed analogous to taking $\msb{a}=1$ and $\msb{b}=0$ in Assumption \ref{hyp:diff:RBnonunif}.
\end{Excont}

\subsection{Uniqueness domain}

We can now state the main result of this section. Recall that, for any $\beta>0$, we have a critical threshold $\zCrit>0$ such that for any $z\in(0,\zCrit)$, the Kirkwood--Salsburg operator $\K_z$ is a contraction in $\Banach$.

\begin{Theorem}\label{thm:diff:uniqueness}
	Let $\Energy$ be an energy functional as in \eqref{eq:diff:energy}, satisfying Assumptions \ref{hyp:diff:0b}+\ref{hyp:diff:4}+\ref{hyp:diff:RBnonunif}. For any fixed inverse temperature $\beta > 0$, there exists a critical activity $\zCrit>0$ such that, for any $z\in(0,\zCrit)$, there exists at most one infinite-volume Gibbs point process $P$ in $\Gibbs$.	
\end{Theorem}

\begin{proof}
	Let $\beta > 0$, $z\in(0,\zCrit)$, and consider two Gibbs point processes $P,\hat{P}\in\Gibbs$.
	\begin{enumerate}[label=\emph{(\roman*)}]
		\item 	We know from Proposition \ref{prop:diff:corrKS}, that the correlation functions $\rho^{(P)}$ and $\rho^{(\hat P)}$ both satisfy the Kirkwood--Salsburg equations \ref{eq:diff:KS}.
		\item By assumption, $\rho^{(P)}$ and $\rho^{(\hat P)}$ satisfy a Ruelle bound for the same $\constan$, and are therefore both elements of $\Banach$.
		\item For $z<\zCrit$, \ref{eq:diff:KS} has a unique solution, so that the correlation functions of $\hat P$ -- and therefore its factorial moment measures $(\alpha^{(\hat P)}_N)_N$ -- must coincide with those of $P$. 
		\item For any $N\geq 1$ and any bounded $\Gamma\subset\mathcal{E}$, we compute
	\begin{equation*}
	\begin{split}
		\alpha^{(P)}_N(\Gamma^N) &= \IE\left[\abs{\gamma_\Gamma}\big(\abs{\gamma_\Gamma}-1\big)\dots\big(\abs{\gamma_\Gamma}-N+1\big)\right]\\
		&= \int_{\Gamma^N} \rho_N(\x_1,\dots,\x_N)z^N\sigmab(d\x_1)\dots \sigmab(d\x_N)\\
		&\leq \int_{\Gamma^N} \prod_{i=1}^N(z\constan(\x_i)) \sigmab(d\x_1)\dots \sigmab(d\x_N) = (zc_\Gamma)^N,
	\end{split}
	\end{equation*}
	with $c_\Gamma\defeq \int_{\Gamma} \constan(\x) \sigmab(d\x)<+\infty$, since $\constan$ is locally integrable (under Assumptions \ref{hyp:diff:4}+\ref{hyp:diff:RBnonunif}).
	\end{enumerate}
	We can then conclude that $P = \hat{P}$ (see \cite{Lenard_1973}, and Proposition 4.12 of \cite{last_penrose_2017}); in other words, $\Gibbs=\{P\}$.
\end{proof}

\setcounter{Example}{1}
\begin{Excont}[continued]\label{ex:diff:end}
	Consider here a potential $\phi=\phi_{hc}$ with a pure hard core at some diameter $R>0$, i.e. $\phi_l\equiv 0$.
	Taking $a_0=R$ in the range Assumption \ref{hyp:diff:1} yields a path potential $\Phi$ (stable, with stability constant $\stabconst_\Phi = 0$) of the form
	\begin{equation*}
		\Phi(\x_1,\x_2) = (+\infty)\,\1_{[0,R)}\big(\inf_{s\in[0,1]} \abs{x_1+m_1(s)-x_2-m_2(s)} \big).
	\end{equation*}
	Under this interaction, two Langevin diffusions are forbidden from coming closer than $R$ to each other, at any given time $s\in[0,1]$. 
	
	For such a choice of $\Phi$ -- which satisfies Assumptions \ref{hyp:diff:0}+\ref{hyp:diff:1}+\ref{hyp:diff:2}+\ref{hyp:diff:RB} -- for any inverse temperature $\beta>0$ and for any activity $z<(e\RuelleC)^{-1}$, the Gibbs point process $P^{z,\beta}$ constructed in Theorem \ref{thm:diff:existence} is the unique element of $\GibbsT$.
\end{Excont}

\textbf{\sffamily Acknowledgements.} This work has been funded by Deutsche Forschungsgemeinschaft (DFG) - SFB1294/1 - 318763901 and Deutsch-Franz\"osische Hochschule (DFH) -- DFDK 01-18. The author would like to warmly thank Sylvie R{\oe}lly and Hans Zessin for the many delightfully fruitful discussions on the topic, as well as Suren Poghosyan for the careful reading and insightful comments. He also thanks the anonymous referee whose comments and questions greatly improved the presentation and redaction of this paper.


\end{document}